\documentclass[12pt]{amsart}

\DeclareFontFamily{U}{matha}{\hyphenchar\font45}
\DeclareFontShape{U}{matha}{m}{n}{
  <5> <6> <7> <8> <9> <10> gen * matha
  <10.95> matha10 <12> <14.4> <17.28> <20.74> <24.88> matha12
  }{}
\DeclareSymbolFont{matha}{U}{matha}{m}{n}
\DeclareFontFamily{U}{mathx}{\hyphenchar\font45}
\DeclareFontShape{U}{mathx}{m}{n}{
  <5> <6> <7> <8> <9> <10>
  <10.95> <12> <14.4> <17.28> <20.74> <24.88>
  mathx10
  }{}
\DeclareSymbolFont{mathx}{U}{mathx}{m}{n}

\DeclareMathSymbol{\obot}         {2}{matha}{"6B}
\DeclareMathSymbol{\bigobot}       {1}{mathx}{"CB}

\usepackage[pdfauthor={Congling Qiu}, 
    pdftitle={???}%
  dvips,colorlinks=true]{hyperref}

\hypersetup{
    colorlinks=red,
    citecolor=red,
    filecolor=black,
    linkcolor=red,
    urlcolor=black
}

\usepackage[english]{babel}
\usepackage[toc,page]{appendix}

  \usepackage{multirow}

\usepackage[utf8]{inputenc}

\usepackage{xtab}
\usepackage{amsmath, amssymb%, cancel,verbatim
}
\usepackage{mathrsfs}
\usepackage[all]{xy}
\usepackage{extarrows}

\usepackage{enumerate}
\usepackage{mathtools,booktabs}
\usepackage{color}
\usepackage{cleveref}
\usepackage{epstopdf} 
\usepackage{booktabs}

\numberwithin{equation}{section}

\theoremstyle{plain}
\newtheorem{proposition}{Proposition}[subsection]
\newtheorem{conj}[proposition]{Conjecture}
\newtheorem{cor}[proposition]{Corollary}
\newtheorem{lem}[proposition]{Lemma}
\newtheorem{thm}[proposition]{Theorem}
\newtheorem{prop}[proposition]{Proposition}

\theoremstyle{definition}

\newtheorem{eg}[proposition]{Example}

\theoremstyle{remark}
\newtheorem{rmk}[proposition]{Remark}

\numberwithin{equation}{section}

\newcommand{\BA}{{\mathbb {A}}} 
\newcommand{\BC}{{\mathbb {C}}} 
 \newcommand{\BF}{{\mathbb {F}}}
 \newcommand{\BH}{{\mathbb {H}}}

 \newcommand{\BP}{{\mathbb {P}}}
\newcommand{\BQ}{{\mathbb {Q}}} \newcommand{\BR}{{\mathbb {R}}}

 \newcommand{\BZ}{{\mathbb {Z}}}

\newcommand{\cK}{{\mathcal {K}}} 
 
\newcommand{\cO}{{\mathcal {O}}} 
 \newcommand{\cR}{{\mathcal {R}}}

\newcommand{\RM}{{\mathrm {M}}}

\newcommand{\fa}{{\mathfrak{a}}} \newcommand{\fb}{{\mathfrak{b}}}
\newcommand{\fc}{{\mathfrak{c}}}

\newcommand{\fm}{{\mathfrak{m}}} 
\newcommand{\fo}{{\mathfrak{o}}} \newcommand{\fp}{{\mathfrak{p}}}

 \newcommand{\fA}{{\mathfrak{A}}} 
 \newcommand{\fD}{{\mathfrak{D}}}

 \newcommand{\fN}{{\mathfrak{N}}}
 
 \newcommand{\fR}{{\mathfrak{R}}}

\newcommand{\wt}{\widetilde}\newcommand{\ol}{\overline}
\newcommand{\wh}{\widehat}

\newcommand{\pair}[1]{\langle {#1} \rangle}

\newcommand{\incl}{\hookrightarrow} 
\newcommand{\nrd}{{\mathrm{nrd}}}\newcommand{\rf}{{\mathrm{f}}}

\newcommand{\bsl}{\backslash}

\newcommand{\lb}{\left(} \newcommand{\rb}{\right)}

\newcommand{\Aut}{{\mathrm{Aut}}}

 \newcommand{\Cond}{{\mathrm{Cond}}}

\newcommand{\Cl}{{\mathrm{Cl}}}

\newcommand{\Gal}{{\mathrm{Gal}}} \newcommand{\GL}{{\mathrm{GL}}}

\newcommand{\Hom}{{\mathrm{Hom}}}

\newcommand{\Ind}{{\mathrm{Ind}}}

\newcommand{\ord}{{\mathrm{ord}}} 
\newcommand{\PGL}{{\mathrm{PGL}}} \newcommand{\Pic}{\mathrm{Pic}}

\newcommand{\PSL}{{\mathrm{PSL}}}

\newcommand{\Nm}{{\mathrm{Nm}}} 

\newcommand{\JL}{{\mathrm{JL}}}\newcommand{\ch}{{\mathrm{ch}}}
\newcommand{\gn}{{\mathrm{gn}}}

\newcommand{\ram}{{\mathrm{ram}}}

\newcommand{\SL}{{\mathrm{SL}}}

\newcommand{\SO}{{\mathrm{SO}}}

\newcommand{\St}{{\mathrm{St}}}

\newcommand{\ur}{{\mathrm{ur}}}  
\newcommand{\vol}{{\mathrm{vol}}}

\newcommand\supervisor[1]{\def\@supervisor{#1}}

\newcounter{elno}

\renewcommand{\cong}{\simeq}

\usepackage{amsthm}
\usepackage{thmtools}

   \subjclass[2010]{Primary 	11F70,  11G18.	Secondary  	    11F03,  14C15,      14G35   ,      14H45.}
    
 \author{Congling Qiu}

\begin{document} 
  \title{Finiteness properties for Shimura curves and modified diagonal cycles}
\begin{abstract}
We prove that only finitely many Shimura curves can have gonality bounded by a given number, and we study the computability of this finite set.
Motivated by the relation between hyperellipticity (that is, gonality 2) and the vanishing of the modified diagonal cycle, we conjecture that such vanishing occurs for only finitely many Shimura curves.
We establish several finiteness and classification results toward this conjecture and, as a by-product, obtain explicit examples of curves with vanishing modified diagonal cycles.
Our computations are based on modular form data from the database \textsf{LMFDB}, and some of them are carried out using the computer algebra system \textsf{Sage}.
 \end{abstract} 
\maketitle 
 \tableofcontents

  \section{Introduction} 

    \subsection{}\label{1.0}

 Shimura curves over $\BC$
are   compactified quotients of the complex upper half-plane $\BH$ by congruence   Fuchsian groups (\Cref{SCC}).  
They  are ubiquitous  in modern number theory and many related fields.
For the  subgroup  of matrices in $\SL_2(\BZ)$ that are upper triangular modulo a positive integer   $N$, the Shimura curve is the  famous modular curve $X_0(N)$.

  The   finiteness  of Shimura curves satisfying a given algebro-geometric property  has long been     studied. One pioneering and
prominent work is Ogg's  classification  of   hyperelliptic curves among   $X_0(N)$'s and analogous Shimura curves  \cite{Ogg,Ogg1}. 
We propose to study  the    finiteness   and    
classification  of  Shimura curves  with a   property, closely related to hyperellipticity but more  intricate, that we now elaborate.

 Let $X$ be a curve, always assumed to be connected smooth projective.  
  For a base point $p$ on $X$,
Gross and Schoen \cite{GS}  defined  the modified  diagonal 1-cycle $\Delta_{p}$ on $X^3$ as follows.
 Set 
\begin{alignat*}{3}\Delta_{12}&= \{(x,x,p):x\in X\},\quad \Delta_{23}&&= \{(p,x,x):x\in X\}, \quad
\Delta_{31}& &= \{(x,p,x):x\in X\}, \\
\Delta_{1}&=  \{(x,p,p):x\in X\},\quad \ 
\Delta_{2}&&=   \{(p,x,p):x\in X\},\quad  \
\Delta_{3}& &= \{(p,p,x):x\in X\}.
\end{alignat*}
Then $\Delta_{p}$ is  the    algebraic cycle
\begin{align*}%\label{eq:mod diag}
\Delta-\Delta_{12}-\Delta_{23}-\Delta_{31}+
\Delta_{1}+\Delta_{2}+
\Delta_{3}.
\end{align*} 
Moreover, if $X$ is   hyperelliptic and $p$ is   a Weierstrass point, they proved that $\Delta_{p}$ is torsion under rational equivalence. 
More generally, replacing $p$ by  any degree 1 divisor  $e$ on $X$, we can literally define the   modified  diagonal  cycle $\Delta_{e}$  in the same way (see \cite{Zha10}). 
Then for $X$ of genus at least 2,
$\Delta_{e}$ is torsion under rational equivalence  only if $e$ is a multiple of the canonical divisor. See \cite[Proposition 2.3.2]{QZ1}.
     We call $X$ \textit{critical} if the modified  diagonal  cycle  is torsion for  this $e$.
  The search for non-hyperelliptic critical curves has recently attracted considerable attention   \cite{BLLS,BS,Lat,QZ1,LS}.    
 %two natural algebraic 1-cycles, defined decades  ago, have recently attracted considerable attention.
 
 %First, let  $X_p$ be the  embedding of $X$ into its Jacobian variety with base point $p$. The Ceresa cycle \cite{C} is the algebraic cycle $ X_p-X_p^-$  on the Jacobian variety where $X_p^-$ is the image of $X_p$ by the inverse map on the Jacobian.
 %If $X$ is  hyperelliptic and $p$ is the fixed point of a hyperelliptic involution (i.e., a Weierstrass point), then  the restriction of the inverse map on the Jacobian  variety to $X_p$ is  the hyperelliptic involution. So  $X_p  -X_p^- = 0$.  
   
%More generally, for any degree 1 divisor  $e$ on $X$, we can define the Ceresa cycle and modified  diagonal  cycle  in the same way (see \cite{Zha10}). 

 %The  class of   critical curves  contains hyperelliptic curves as we have seen above. 
 %On the other hand,               

 When specialized to Shimura curves, the modified  diagonal  cycle is closely related to $L$-functions and automorphic forms through the generalized Gross--Kudla conjecture \cite{GK,YZZ}. 
This also falls into the scope of the more general arithmetic Gan--Gross--Prasad conjecture \cite{GGP,Zha12}. 
Using this relation, critical Shimura curves can be used to find $L$-functions with vanishing orders  larger than 1 \cite{Qiuvan}.

  \subsection{} \label{1.1}   
  By  \cite[Theorem 1.5.5]{Zha10}, $X$ is critical if and only if 
  the Ceresa cycle  $[X]-(-1)^\ast[X] $ on
  the Jacobian variety of $X$  is torsion, where we embed $X$ into $J_X$ by $x\mapsto x-e$. 
  Then by a classical result of Ceresa  \cite{C}, in genus at least 3,    a general curve  over $\BC$ is not critical.   
This is also a well-known fact  for hyperellipticty.  
   Since  Shimura curves can be concretely constructed using congruence subgroups, we are particularly interested in the computational 
  aspects of such common rareness of hyperelliptcity and  criticalnes

  %Actually, in genus at least 2, this heppensIn this case, 
 %\cite[Theorem 1.5.5]{Zha10}, $X$ is critical if and only if 
 % the Ceresa cycle  $[X]-(-1)^\ast[X] $ on
%  the Jacobian variety of $X$  is toraion, where we embed $X$ into $J_X$ by $x\mapsto x-e$.
 %So by the result of Ceresa  \cite{C},  in genus at least 3,    a general curve  over $\BC$ is not critical.   
  
%Recall  that a Shimura curve is associated with  a unique quaternion algebra over a totally real number field \cite[Definition 1]{Tak1}. 
%For example, classical modular curves  are associated with  $\RM_{2,\BQ}$.
  \begin{thm}\label{thm:hec} 
   There
exist only finitely many congruence  Fuchsian groups   up to  conjugation whose  associated 
Shimura curves are hyperelliptic. 
Moreover, this finite set is contained in a  computable 
 set of  explicit  arithmetic Fuchsian groups. 
 Here, an  explicit   arithmetic Fuchsian groups is given by an  explicit  finite set of generators  and an  explicit  finite set of relations.

  \end{thm} 
    We actually prove the same result for curves of a bounded gonality (\Cref{thm:gn}).
Under the theorem, to compute this finite set of congruence  Fuchsian groups, one way is to determine whether an explicit  arithmetic Fuchsian group is congruence and  whether  the associated 
Shimura curve is hyperelliptic. However, these two problems are out of the scope of this paper.

 In \cite{QZ1}, we   conjectured that  hyperelliptic curves are the ``majority" of critical curves.
 This leads us to believe the same finiteness of   critical Shimura curves. 

    \begin{conj}\label{conj:diag}
There
exist only finitely many    critical Shimura curves up to isomorphism.
 Moreover, the finite set of  critical Shimura curves is    computable.
 
  \end{conj} 
   
   Since the first version of this paper was released, progress has been made on \Cref{conj:diag} for modular curves in \cite{KLQY,LR}.
   
    \subsection{}\label{1.3} 
 It is usually hard to verify the criticalness  directly. %An equivalent condition was given in \cite{SV}, and  was used to provide curves with vanishing modified diagonal cycles in \cite{Lat}.  
  In  our previous work on  modified diagonal cycles \cite{QZ1},  we considered curves with many automorphism and gave a sufficient  representation-theoretical criterion for criticalness.  In  another  work  \cite{QZ2}, this  criterion  was extended to  certain Shimura curves  using automorphic representations  as follows.         Let $F$ be a totally real number field.
      A quaternion algebra  $B$   over   $F$ is called  {\em almost definite}
if  $B$ is  split at  one archimedean place $\tau:F\incl \BR$ of $F$ and division at all other archimedean places of $F$.
Let  $B^+\subset B^\times$ be the subgroup of elements with totally positive norms (equivalently, with positive determinants at $\tau$).
Then $ B^+\subset B_\tau^\times\cong \GL_2(\BR)$ acts on $\BH$ by linear fractional transformations with $ F^\times\subset B^+$ acting trivially. 
  Consider a    subgroup $\Gamma\subset B^+/ F^\times$  with
     an order $\cO$ of $B$  such that the image of  $\cO \cap B^+ $ in $B^+/ F^\times$
   is of finite index in $\Gamma$. 
 In this case, the smooth compactified quotient $X = \tau (\Gamma)\bsl \BH^*$  is   a   Shimura curve over $\BC$,
and   $H^{1,0}(X)$ generates  (and is  a subspace   of) a direct sum $\pi_X$ of  automorphic representations  of $B^\times$.  See \Cref{Good  curves} for  more details.
 We call  $X$ \textit{auto-critical} if
there is no nonzero diagonal invariant trilinear forms on $\pi_X$ (\Cref{Good  curves}).     If $X$ is \textit{auto-critical}, then 
  $X$ is critical (\Cref{vancor2}).  
  
   \begin{conj}\label{conj:tril}
There
exist only finitely many   auto-critical   Shimura curves up to isomorphism.
Moreover, the finite set of  auto-critical   Shimura curves is    computable.
   
\end{conj}

     \begin{thm} \label{thm:finiteintro}
  
 The  set of   quaternion algebras  associated to auto-critical   Shimura curves  is  finite.
 Moreover, this finite set is contained in a  computable 
 set of  explicit  almost-definite quaternion algebras over totally real fields.

    \end{thm} 
  \Cref{thm:finiteintro}    will be proved in the end of \Cref{Good  curves}. 
  
When $F=\BQ$,
we give  classifications of auto-critical Shimura  curves with  classical level structures  in \Cref  {thm:auto-critical}, 
\Cref  {thm:Shauto-critical} and \Cref  {thm:ShAauto-critical}.    In particular, we obtain some  examples of critical curves.  
To prove the classification, we follow an algorithm based on the LMFDB database in \Cref{newforms in the database LMFDB}. On top of this, while most of the computations are "by hand",  some of then are carried out using the computer algebra system Sage. 
        \subsection{}        
In Section 2, we prove the 
finitenss of Shimura curves with  bounded  gonality.
In Section 3, we study the
finitenss of auto-critical Shimura curves.
In Section 4, we classifify  auto-critical Shimura curves over $\BQ$.

    \subsection*{Acknowledgments}
The author  would like to thank   Pete~L.~Clark,  David~Loeffler, Adam Logan, Bjorn~Poonen,  John~Voight, Jared~Weinstein, 
Shouwu~Zhang  and Wei~Zhang for their comments and helpful discussions.

 \section{Finitenss   of  Shimura curves with  bounded gonality}  %
 \addtocontents{toc}{\protect\setcounter{tocdepth}{2}}
 
 In this section,
 we   recall the notion and properties  of Shimura curves over $\BC$, and proves some   finiteness   results related to gonality.

  \subsection{Quaternion algebras}
  Shimura curves are defined in terms of quaternion algebras. So we study the latter first.
We      refer   to  \cite{Voi2} for the encyclopedia  of  quaternion algebras, and \cite[Section 1]{KV} for a quick introduction.
Let  $F$ be a number field and let $R=\cO_F$ (following notations in    \cite{Voi2}).
When an order $O$ in $B$ over a number field $F$ is referred,   we understand that a set of generators   of $O$ as an  $R$-module  in $B$ is given explicitly.

    \begin{prop}\label{prop:Order} 
For a quaternion algebra $B$ over a number field,   the set of maximal orders in $B$ up to conjugation by $B^\times$  is finite and computable.

 \end{prop}
   \begin{proof} 
Since these maximal orders are all locally isomorphic, by 
\cite[Lemma 17.4.13]{Voi2}, we only need determine the  (right) class set of one maximal order. 
First, by  \cite[Theorem 7.14]{Voi1}, we can  compute one maximal order.  Second, by  \cite[Algorithm 4.4]{KV}, 
the  (right) class set can be computed. The proposition is proved.
   \end{proof}
   
Let  $\Pic_R(O)$ be the   group of  invertible  fractional two sided $O$-ideals modulo   the   group of  principal fractional  $R$-ideals  \cite[(18.4.6)]{Voi2}.

              \begin{lem}[{\cite[Corollary 3.3]{KV}}]\label{lem:Pic} 
           For a  maximal order  $O$,   a set of
representatives for the two-sided invertible ideal classes of $O$ is computable. 

 \end{lem}

    \begin{prop}\label{prop:Pic} 
For any order $O$, we can compute a set of
representatives for the two-sided invertible ideal classes of $O$. In particular, $|\Pic_R(O)|$ is computable.

 \end{prop}
   \begin{proof} 
   
By  \cite[Theorem 7.14]{Voi1}, we can  compute one maximal order $O'$ containing $O$.
The rest of the proof is the same as \cite[Proposition 18.4.10]{Voi2}, with the finiteness of $\Pic_R(O')$ replaced by \Cref {lem:Pic}.
   \end{proof}
   The following lemma was advised by John~Voight.
  \begin{lem}\label{lem:ordgen} 
     For any order $O$ and a finitely generated subgroup  $G\subset O^\times$ with an explicit finite  set $S$ of generators, the order generated by $G$ is computable.
     
  \end{lem}  
   \begin{proof} 
In order for $G\subset O^\times$ to be of finite index,     $S$ must contain at least 2 non-commutative elements, say $X,Y$. 
Then $F + FX + FY$ is  already 3-dim in $B$ and it can not be closed under multiplication since $B$ has no such subalgebras. 
So one of $X^2, Y^2, XY, YX$ must not in $F + FX + FY$. Add it  and the remaining generators of $S$
to $R+ RX + RY$ to get an $R$-lattice $L\subset O$ of finite index (additively). Compute a set of generators   of $L$ as an  $R$-module. Keep adding the intersection terms of the generators  to $L$. This process stops after finite steps and we get an $R$-lattice in $O$ closed under multiplication. This is the     order generated by $G$. \end{proof}

 Let $B^1$ be the group of elements of norm 1 in $B$. 
Let    $PB^\times = B^\times/F^\times $ and  $PB^1 = B^\times/\{\pm 1\} $
        
    A quaternion algebra  $B$   over  a totally real  number field $F$ is called  {\em almost definite}
if  $B$ is  split at one archimedean place $\tau:F\incl \BR$ of $F$ and division at all other archimedean places of $F$.

    \begin{prop}\label{prop:Norm} 
Let $B$  be  a definite or almost definite quaternion algebra over a totally real  number field  $F$. For an order $O$ in $B$, the normalizer 
$$N_{PB^1}(O) : = \{b\in B^1: b^{-1}O b = O\}$$ of $O$ in $PB^1$ is finitely presented and  computable (i.e., a finite set of generators and a finite set of  relations defining $N_{PB^1}(O)$ is computable).  Moreover, the index of $(O\cap B^1) /\{\pm 1\}$ in  $N_{PB^1}(O)$  has an explicit upper bound.

 \end{prop}  
   \begin{proof} 
   Let $B^0\subset B$ be the subspace of elements of reduced trace 0  and $O^0 = O\cap B^0$.
Compute a  $\BZ$-basis of $O^0$, and thus realize $\GL(B)$ as $\GL_n$ over $\BQ$ with $n = 3[F:\BQ]$.
   We will  apply     \cite{GrunSe}  by finding an \textit{explicitly given}
 $\BQ$-subgroup  
of $G$ of  $\GL_n$   (i.e., defined by explicit $\BQ$-polynomials, see the first page of \cite{GrunSe}), and
   realizing $N_{PB^1}(O)$  as a finite index subgroup  $G\cap  \GL_n(\BZ)$.
   
   The reduced norm on $B$ makes $V$ a quadratic space. The action of $a\in PB^\times$ on $x\in B $ by $axa^{-1}$  identifies
$PB^\times$ as $G := SO(V)$. See \cite[Proposition 4.5.10]{Voi2}. Moreover, by definition, this identification can be explicitly computed.
Claim:
$N_{PB^1}(O)$ is a finite index subgroup  $G\cap  \GL_n(\BZ)$ with  an explicit upper bound for this index.
This is  proved in two steps with the middleman  $N_{PB^\times}(  O )$,   the  normalizer of $O$    in $PB^\times$.
First, we have $$G\cap  \GL_n(\BZ) = N_{PB^\times}(O^0) = N_{PB^\times}(R\oplus O^0) ,$$
  the corresponding normalizers    in $PB^\times$.  
  On $B$,   the reduced norm defines a quadratic form, and for an $R$-lattice $L$, let $L^\vee$ be the dual lattice.  
  Then $N_{PB^\times}(L)=N_{PB^\times}(L^\vee)$. 
  Consider  the action of $N_{PB^\times}(R\oplus O^0)$ on $   (R\oplus O^0)^\vee/O^\vee $.
  Let $$Q : =  (R\oplus O^0)^\vee/O^\vee \cong O/ (R\oplus O^0).$$ Then we have  an exact sequence
     $$1\to N_{PB^\times}(  O ) \to N_{PB^\times}(R\oplus O^0) \to \GL(Q).$$  
Since the  reduced trace gives 
$Q\incl R/2R\cong \BZ/2\BZ^{\oplus [F:\BQ]}$.
So  $N_{PB^\times}(  O )$   is a finite index subgroup  $G\cap  \GL_n(\BZ)$ with  an explicit upper bound for this index.
Second, let $l$ be the order of 
   $  N_{PB^\times}(O)/(  O^\times/ R^\times)$
   and $m$ the index of $(O\cap B^1)/\{\pm 1\}  \incl   O^\times/ R^\times$. Then by the injective homomorphism 
   \begin{equation}N_{PB^1}(O)/(  (O\cap B^1) /\{\pm 1\} ) \incl N_{PB^\times}(O)/(  O^\times/ R^\times),\label{NPB1O}\end{equation} 
     the index of $N_{PB^1}(O) \incl N_{PB^\times}(O)$  is  bounded by $lm$.
By  \Cref{prop:Pic} and the injection  $N_{PB^\times}(O)/(  O^\times/ R^\times)\incl \Pic_R(O)$  in \cite[(18.5.5)]{Voi2}, $l$ is computable. 
 By taking reduced norm, we have $m\leq |R^\times/\{a^2: a\in R^\times\}|$ which can be explicit bounded using Dirichlet's unit theorem.
 The claim is thus proved.  Moreover, by \eqref{NPB1O}, 
 $[N_{PB^1}(O):  (O\cap B^1) /\{\pm 1\} ]\leq l.$ This proves the second claim in the proposition.
 
 Now we want to apply   \cite{GrunSe}.  We need to verify the three conditions at the bottom of the first page and the top of the second page of loc. cit.. 
The first two are the claim in the last paragraph. The third amount to 
  determining if any $\alpha\in B^\times$ has norm being a square in $F^\times$.  It is well-known that we can factor polynomials over number fields, so 
the third  condition holds. 
Thus we can apply
 \cite[ALGORITHM B., p 533]{GrunSe}  to  compute a  finite generating set $X$ of $N_{PB^1}(O)$.
Finally, to compute a  finite set of defining relations for $N_{PB^1}(O)$  on the generating set $X$, we  use that $B$ is definite or almost definite. In the first case $G_\BR$ is connected, and in the second case,
  by the explicit identification of $PB^\times$ and $G$, the connected component of $G_\BR$ is given by the norm 1 condition.  Thus by \cite[Remark, p 582]{GrunSe}, we can compute a  finite set of defining relations for $N_{PB^1}(O)$  on the generating set $X$.   The proposition is thus proved.  
   \end{proof}

  \subsection{Shimura curves   over $\BC$}\label{SCC}

 A  discrete   subgroup  $\Gamma\subset \PSL_2(\BR)$  (i.e., a Fuchsian group) is called {\em arithmetic}   if 
 there is an almost definite   quaternion algebra  $B$ as above, so that $B^1_\tau\cong \SL_2(\BR)$   and  an order $\cO$ of $B$ such that   $(\tau (\cO) \cap \SL_2(\BR))/\{\pm1\}$  is  commensurable to  $\Gamma$. 
 An  arithmetic subgroup $\Gamma\subset \PSL_2(\BR)$ is called  a
{\em congruence}  subgroup if   $(\tau (\cO) \cap \SL_2(\BR))/\{\pm1\}\subset \Gamma$ for some order $\cO$ of $B$. 
In this case,   the smooth compactified quotient $\Gamma\bsl \BH ^* $  is called a   Shimura curve over $\BC$. 
 The Shimura curves over $\BC$ defined in  \Cref{1.3}  fall in this class by the embedding
  $\tau( B^+/F^\times)\subset \PSL_2(\BR)$.

 Now we consider gonality of Shimura curves. For a curve (smooth projective   connected) $C$ over a field $k$, its gonality  $\gn(C)$ is the lowest degree of a nonconstant rational map from $C$ to $\BP^1$.

\begin{thm}\label{thm:gn}  
    There
exist only finitely many  congruence  subgroups $\Gamma\subset \PSL_2(\BR)$ up to conjugation  
%check  A characterization of arithmetic Fuchsian groups By Kisao TAKEUCHI, Or can atleast say (B,F,Gamma)
such that the
Shimura curves  $X_\Gamma$ have gonalities  bounded by a given number. 
Moreover, this finite set is contained in a  computable 
 set of  explicit  arithmetic Fuchsian groups. 
  Here, an  explicit   arithmetic Fuchsian group is given by an  explicit  finite set of generators  and an  explicit  finite set of relations.
 
   \end{thm}

   For the proof we need some preparations. 
  \begin{prop}[Hermite--Minkowski]\label{prop:HM} 
  For   real  numbers $d,M$, the set of number fields   with  degree  bounded  by $d$ and discriminant bounded   by $M$ is finite and  computable.
  
 \end{prop}
   \begin{proof} The standard proof of the finiteness part using Minkowski's Lattice Point Theorem produces  an explicit  uniform bound on the coefficients of the minimal polynomial of  a generator of such a number field. The lemma follows.
   \end{proof}
      
  The hyperbolic volume of $\Gamma\bsl \BH$ is finite for an  arithmetic (in particular, congruence) Fuchsian group $\Gamma\subset \PSL2(\BR)$, and is called the covolume of $\Gamma$.  
  \begin{thm}\label{thm:finvol} 
 There
exist only finitely many arithmetic Fuchsian groups $\Gamma$  up to  conjugation whose covolumes are bounded by a given number. 
Moreover, this finite set is contained in a  computable 
 set of  explicit  arithmetic Fuchsian groups. 
 
  \end{thm}
   \begin{proof} The finiteness follows from
   Takeuchi 
\cite[(2.3), Theorem 2.1]{Tak1}.  We prove the    computability by refining the proof in loc. cit.

Recall that such $\Gamma$ is associated with a pair $(F,B)$ as above. 
   By the proof on \cite[p 384]{Tak1} on the finiteness of such $F$ and \Cref{prop:HM}, the set of such $F$  is    computable.
After fixing $F$, by the proof on the finiteness of such $B$ in loc. cit., the set of such $B$  is    computable.
Then by \Cref{prop:Order},   the set of maximal orders in $B$ up to conjugation by $B^\times$ is    computable.

After computing a maximal order $O'$ of $B$, by \cite[p 384]{Tak1}, there is a   subgroup $\Gamma^{(2)}$ of $O'\cap B^1$ of an explicitly bounded index such that
$\Gamma^{(2)}\{\pm1\}\subset \Gamma \{\pm1\} \subset N_{PB^1}(\Gamma^{(2)})$. By the low index subgroup algorithm
   \cite{DS,Sim}, we can enumerate  an explicit finite list of possibilities of $\Gamma^{(2)}$. And we fix one of them.
   Let $O$ be   the order generated by $\Gamma^{(2)}$, which is computable by \Cref{lem:ordgen}.
Then we have a commutative  diagram of injections
   \begin{equation} \xymatrix{  \Gamma^{(2)}\{\pm1\}  \ar[r]\ar[d] &N_{PB^1}(\Gamma^{(2)})\ar[d] \\
(O\cap B^1) /\{\pm 1\} \ar[r]
&N_{PB^1}(O^\times).}
 \end{equation}
 By the second part of \Cref{prop:Norm} and that 
the index of $\Gamma^{(2)}\{\pm1\}$ in $(O\cap B^1) /\{\pm 1\}$ is explicitly bounded, the index of  
$\Gamma^{(2)}\{\pm1\}$  in $N_{PB^1}(O^\times)$  is explicitly bounded. 
So is the index of $\Gamma$.
 By the first part of \Cref{prop:Norm}, $N_{PB^1}(O^\times)$ is  finitely presented and  computable.
Then the low index subgroup algorithm
   \cite{DS,Sim} produces an explicit finite list of possibilities of $\Gamma$. The theorem is proved.
    \end{proof}

   Fix $\lambda\in (0,1/4]$ such that for any  congruence arithmetic Fuchsian group 
$\Gamma$,
the first nonzero eigenvalue of the Laplacian on $X_\Gamma:=\Gamma\bsl \BH^*$    with respect to the hyperbolic metric is at least $\lambda$. 
 For example, we have the classical lower bound $3/16$ by Selberg (and  the Jacquet--Langlands correspondence, see \cite{Vig}, see also  \cite[(5)]{LRS} for an improvement of  $3/16$). 

  \begin{proof} [Proof of \Cref{thm:gn}]

 Li and Yau \cite{LY} gave a linear lower   bound  for  gonality of a curve $X/\BC$ in terms of the volume of $X$ and
the first non-zero eigenvalue $\lambda_1(X)$ of the Laplacian
 $$ \gn(X)\geq \frac{\lambda_1(X)}{2\pi }\vol(X).$$
 In the case $X_\Gamma$ where $\Gamma$ is congruence,   $\lambda_1$ may be  uniformly replaced by $\lambda$. 
 So  \Cref{thm:finvol}  implies the  theorem. 
\end{proof}

\begin{rmk}
  (1)   Unlike  \Cref{thm:finvol}, the finitness in \Cref{thm:gn} can not be true for  arithmetic Fuchsian groups. 
  As pointed out to us by  S.~Zhang, by Belyi's theorem, any curve over $\ol\BQ$ is  of the form $  \Gamma\bsl \BH^*$  for some $\Gamma$  of finite index in $\SL_2(\BZ)$ (thus arithmetic).  
  As there are hyperelliptic curves over $\ol\BQ$ of arbitrary large genus,   both   Theorem \ref{thm:gn} and Conjecture \ref{conj:diag} fail if we replace ``congruence" by ``arithmetic".

(2) There is 
 a similar finiteness property  for genus,  by
 Long, Machlachlan, and Reid  \cite{LMR}.  (They only stated their result for genus 0, but the argument works in general.)
 Besides the above finitness result of Takeuchi and  the Selberg lower   bound, they also used   a remarkable result of  Zograf  relating the genus of  $\Gamma\bsl \BH$ to the covolume of  $\Gamma$.
 We are obliged to mention that 
  \Cref{thm:gn} is also implied by this finiteness  for genus,  and  the lower bound  on the gonality in terms of genus for  congruence  quotients of $\BH$ by  Abramovich \cite{Abr} (see for example \eqref{Xgamma}).
  But this  implication is really a detour, as 
  the gonality   bound is proved using Li and Yau \cite{LY}.
  
  (3) The finiteness and algorithm given in    \Cref{thm:gn} is obviously very  inefficient. When specialized to Shimura curves with  classical levels as defined in the next section, we get   better  results.  For example,  \Cref{lem:Qo1} and its proof, as well as \Cref{eqgon}.  \end{rmk}

 \section{Finitenss   of auto-critical Shimura curves}  %
 
\addtocontents{toc}{\protect\setcounter{tocdepth}{2}}
 In this section,
 we first recall the notion and properties  of Shimura curves over a totally real number field.
 Then we define auto-critical Shimura curves and  prove the related  finiteness result.

    \subsection{Shimura curves over a totally real number field}\label{1.4}
 
  Let $F$ be a totally real number field, $\BA$  its ring of adeles and $\BA_\rf$  its ring of finite adeles.
  Let $B$  be an almost definite   quaternion algebra    over $F$,  split at one archimedean place $\tau:F\incl \BR$ of $F$ and division at all other archimedean places of $F$.

  For an open compact-modulo-center subgroup $K$ of $B^\times(\BA_\rf)$, we have the smooth   compactified   Shimura curve $ X_K$ for  $B^\times$ of level $K$ over $F$ \cite[3.1.4]{YZZ}.
   The  complex uniformization  of  $X_K$ via the distinguished archimedean place $\tau: F\incl \BR\subset \BC$ is given by
\begin{equation}X_{K,\BC} \cong B^+ \bsl \BH \times B^\times(\BA_\rf)/K\coprod \{\text{cusps}\}.\label{Xcomplex}\end{equation}
Here $B^+\subset B^\times$ is the subgroup of elements with totally positive norms (equivalently, with positive norms at $\tau$),  and 
 the cusps exist if and only if $X_K $ is a modular curve, i.e. $F=\BQ$, and $B $ is the matrix algebra.

Let us look at some classical examples.
For a non-archimedean local field $E$, let   $\fm_E$ be the maximal ideal of  the ring of integers $\cO_E$. 
For $n\geq 0$, let  
\begin{equation*} 
\Gamma_0\lb\fm_E^n\rb=\left\{
\begin{bmatrix}
a & b \\
c & d 
\end{bmatrix}\in \GL_2(\cO_E):c\in \fm_E^n\right\},
\end{equation*}
\begin{equation*}
\Gamma_1\lb\fm_E^n\rb=\left\{
\begin{bmatrix}
a & b \\
c & d 
\end{bmatrix}\in \GL_2(\cO_E):c\in \fm_E^n,d\in 1+ \fm_E^n\right\},
\end{equation*}
\begin{equation*}
\Gamma\lb\fm_E^n\rb=\left\{
\begin{bmatrix}
a & b \\
c & d 
\end{bmatrix}\in \GL_2(\cO_E):b,c\in \fm_E^n,\ a,d\in 1+\fm_E^n\right\}.
\end{equation*}

     Let $S_B$ be the finite set of finite places $v$ of $F$ such that $B_{v}:=B(F_v)$ is a division quaternion algebra.  
     For $v\not\in S_B$, $B_v^\times\cong \GL_2(F_v)$.   Then $\GL_2\lb \cO_{F_v}\rb $ is its unique maximal  open compact subgroup up to conjugation.  
  For $v\in S_B$, let $\cO_{B_v}$ be the unique maximal order of   $B_{v}$ and $\fm_{B_v}:=\cO_{B_v}\bsl \cO_{B_v}^\times$ its unique maximal (two-sided)   ideal. Then $\cO_{B_{v }}^\times $ is the unique maximal  open compact subgroup of $B_v^\times$.

                 For an ideal $\fN$ of $\cO_F$ coprime to all $v\in S_B$,  and an ideal $\fA$ of $\cO_F$ such that $\ord_v(\fA)=0,v\not \in S_B$ and $\ord_v(\fA)\geq 1,v \in S_B$,   
      let $$ X_*^\fA(\fN) =X_K, Y_*^\fA(\fN) =X_{K\BA_\rf^\times} \text{ for } *\in \{0,1,\emptyset\}$$ where 
      $$K=K_*^\fA(\fN)$$ is defined by
         $$K_{v}=\Gamma_*\lb\fN\cO_{F_v} \rb \text{ for } v\not
      \in S_B
      ,$$
      and  \begin{equation}\label{KvA}
K_{v}=1+\fm_{B_{v }}^ {\ord_v(\fA)-1} \text{ for }v\in S_B.
 \end{equation} 
  If $\fA$ is square-free, i.e., $\fA$  is the discriminant of $B$,  then $K$ is the closure of the so-called Eichler order of level $\fN$ (see for example \cite{Cla,Voi})  in $B^\times(\BA_\rf)$.   
  In this case, we also  let  $$ X_*^B(\fN) =X_*^\fA(\fN), Y_*^B(\fN) =Y_*^\fA(\fN).$$
    If $B = \RM_{2,\BQ}$,   then we have
 the classical modular curves.

We have some remarks about $Y_*^\fA(\fN) $. First, in the more standard terminology, it is a Shimura curve for the reductive group $PB^\times = B^\times/F^\times $ (while $X_*^\fA(\fN) $ is for $B^\times$). 
Second,  we have a natural map $X_*^\fA(\fN) \to Y_*^\fA(\fN) $ and 
$$  Y_*^\fA(\fN) =X_*^\fA(\fN) /
 \BA_\rf^\times
 $$  where  $\BA_\rf^\times$  acts on any $X_K$  by right multiplication via \eqref{Xcomplex}, and the action factors through a finite quotient. 
Finally, by definition, it is easy to see that $$Y_0^\fA (\fN)=Y_1^\fA(\fN).$$

 Another class of classical examples are Atkin-Lehner quotients, which will be discussed in the next subsection.

  %\subsection{Geometrically connected components}\label{Geometrically connected components}
The connected components of $X_{K,\BC}$ are exactly  the Shimura curves over $\BC$ defined in  \Cref{1.3}. Let 
us study some examples.

Let  $\wh \cO\subset \BA_\rf$ the ring of integral adeles.
For an ideal $\fm$ of $\cO$, let  $\wh\cO^\times(\fm)=(1+\fm \wh \cO)\cap \wh \cO^\times$. Then
ray class group, resp. narrow ray class group of $F$ associated to the modulus $\fm$ is $$\Cl_\fm\cong F^\times\bsl \BA_\rf^\times/ \wh\cO^\times(\fm),\text{ resp. } \Cl_\fm^+\cong F^+\bsl \BA_\rf^\times/\wh\cO^\times(\fm).$$

         For $\fA$ as above, let     $$\fm=\prod_{\fp|\fA}\fp^{[\frac{\ord_\fp(\fA)}{2}] }.$$
For $*\in \{0,1\}$ (but not $\emptyset$), a direct computation shows that the zeroth fundamental groups  
  \begin{align}\label{eq:pi0} 
  \pi_0\lb X^\fA_*(\fN)_\BC\rb\cong \Cl_\fm^+,\  \pi_0\lb Y^\fA_*(\fN)_\BC\rb\cong\Cl_\fm^+/ \lb \Cl_\fm^+\rb^2.
     \end{align}
   
   Below, we only need $*=0$.  Let $K=K_0^\fA(\fN)$.
   Let $X^o$ be a connected component of $X_0^\fA(\fN)_\BC$ and $Y^o$ its image in $Y_0^\fA(\fN)_\BC$ (which is also a connected component). 
   \begin{lem}\label{lem:degXY}
 The degree of $X^o\to Y^o$  divides the cardinality  $|\Cl_\fm |$. 
 \end{lem}
  \begin{proof}  Outside cusps,
  we may take $X^o\to Y^o$  to be
$$(B^+\cap K)\bsl \BH\to (B^+\cap (K\BA_\rf^\times))\bsl \BH.$$
So the degree is the cardinality of  
$(B^+\cap (K\BA_\rf^\times))/(B^+\cap K)F^\times$, which is a subgroup of $$   (K\BA_\rf^\times)/ ( KF^\times)\cong  \BA_\rf^\times/ ( KF^\times\cap \BA_\rf^\times)=\BA_\rf^\times/\lb \wh \cO^\times(\fm) F^\times\rb\cong \Cl_\fm.$$
   \end{proof}
  \begin{cor}\label{lem:YKXK} If $|\Cl_\fm |=1$, then $X_0^\fA(\fN)\to    Y_0^\fA(\fN)$  is an isomorphism.
 \end{cor}
   \begin{proof}          By \Cref{lem:degXY}, we only need  $|\pi_0\lb X_*(\fN)_\BC\rb|=|\pi_0\lb Y_*(\fN)_\BC\rb|. $
   But  the kernel of the natural homomorphism from the narrow class group to the   class group 
   is $2$-torsion, so the equality follows from \eqref{eq:pi0}.        
    \end{proof}
      \begin{rmk}  Here is another more conceptual argument for the isomorphism when either curve has  genus at least $2$.
                     The holomorphic differential $1$-forms on $X_0^\fA(\fN)_\BC$  can be realized as some automorphic forms  whose central characters  are characters of the   class group of $F$  associated to $\fm$. 
                      And the ones on $ Y_0^\fA(\fN)_\BC$  are those  whose central characters are trivial after pullback to $X_0^\fA(\fN)_\BC$.                       If the ray class group  is trivial, $ Y_0^\fA(\fN)_\BC$ has the same  space of  holomorphic differential forms with  $X_0^\fA(\fN)_\BC$  via pullback.   
  
  \end{rmk}
  Now we want to clarify the relation between connected components of $X_{K,\BC}$ and the Shimura curves over $\BC$ defined in
\Cref  {SCC} for later use, also because they are often confused in literature.  
Recall that  $B^1$ be the group of elements of norm 1 in $B$. 
           Let   $$ X^1=(B^1\cap K)\bsl \BH^* \text{ resp. }  Y^1=(B^1\cap K\BA_\rf^\times)\bsl \BH^*.$$    
   Let $X^o$ be a connected component of $X_{K,\BC}$ and $Y^o$ its image in $Y_{K,\BC}$ (which is also a connected component). 
Still let $K=K_0^\fA(\fN)$.
Assume that $\fA$ is square-free (i.e., $\fA$ is the discriminant of $B$).
 Then  $\Cl_\fm$ becomes the usual class group $\Cl$ of $F$.
  \begin{lem}\label{lem:degX1X}
 The degree of the  natural map $X^1\to X^o$  divides $| \cO^{\times,+} /\lb  \cO^\times\rb^2|$. 
 
 %(2)   If $|\Cl|=1$,  the degree of  the  natural map  $Y^1\to Y^o$  divides $| \cO^{\times,+} /\lb  \cO^\times\rb^2|$. 
  
 \end{lem}
  \begin{proof}  The degree is the cardinality of 
$G: = (B^+\cap K)/\lb (B^1\cap K)(F^\times\cap K)\rb.$
Note that  $F^\times\cap K=\cO^\times$.
By the reduced norm map on $B^\times$,
$G$ is a subgroup of    $  \cO^{\times,+} /\lb  \cO^\times\rb^2.$
  \end{proof}
 
\begin{lem}[{\cite[Corollary 2.2., Proposition 2.4]{EMP}}]\label{lemclassn}
   Assume  \begin{itemize}
    \item $F$ has odd narrow class number, or
    \item $F$ is a subfield of $\BQ(\zeta_p)$  for a prime number $p$ and has odd  class number.
  \end{itemize}
  Then   $  \cO^{\times,+} =\lb  \cO^\times\rb^2.$ 
\end{lem} 

  \begin{cor} Assume  \begin{itemize}
    \item $F$ has   narrow class number 1, or
    \item $F$ is a subfield of $\BQ(\zeta_p)$  for a prime number $p$ and has  class number 1.
  \end{itemize}
      Then
  $X^1, X^o$,  $Y^1$ and $Y^o$ are all isomorphic. 
    \end{cor}
          
     \subsection{Atkin--Lehner quotients}\label{Atkin--Lehner quotients I}

 Let  us define the  Atkin--Lehner involutions on $ Y_0^B(\fN)$. 
 
 Let $ K = K_0^\fA(\fN)$.
 First, at $v\in S_B$, $B_v^\times$ normalizes $K_v$ so that $B_v^\times/K_v$  acts on
  $X_K$ by  the right multiplication isomorphism $\rho(g) :  X_{ K}=  X_{g Kg^{-1}}\to X_{  K}$ \cite[3.1.4]{YZZ}, 
for $g\in B_v^\times/K_v$. 
The action factors through a finite quotient as $\Aut(X_K)$ is finite.
The same action on  $ Y_0^\fA(\fN) = X_{K\BA_f^\times} $   factors through $B_v^\times/F_{v}^\times K_v$. If moreover, $\ord_v\fA=1 ,$  then
   \begin{equation}\label{DZ/2'}   B_v^\times/  K_v \cong \BZ\text{ resp. }B_v^\times/F_{v}^\times K_v \cong \BZ/2.
    \end{equation}  %  (see \eqref{DZ/2}).
In this case,
let $w_v$ be the generator of $\BZ/2$, which is  the Atkin--Lehner involution of $X_{K\BA_f^\times} $  at $v$.
    Next, 
      we consider Atkin--Lehner involutions  over the finite set  $S_\fN$   of finite places 
       of $F$  dividing $\fN$. For $v\in S_\fN$, 
            let 
             \begin{equation}\label{ALA1}
             g_v\in \begin{bmatrix}
0 & 1 \\
\fm_{F_v}^{m-1}\bsl \fm_{F_v}^{m} & 0.
\end{bmatrix},
\end{equation}
 where $m=\ord_v\fN$. 
 Then 
  $K_v= g_vK_vg_v^{-1}$.  
Let $w_v $ be
 the right multiplication isomorphism $\rho(g) :  X_{K\BA_f^\times}=  X_{gK\BA_f^\times g^{-1}}\to X_{K\BA_f^\times}$, which has order $2$ as $g^2\in F_v^\times$. This is   the Atkin--Lehner involution of $X_{K\BA_f^\times} $ at $v$.    
 Put them together, if $\fA$ is square-free ((i.e., $\fA$ is the discriminant of $B$), we have the set 
 $$\{w_v:v\in S_B\coprod S_\fN\}$$  of  Atkin--Lehner involutions on $ Y_0^B(\fN)$.
In fact, if $F=\BQ $, this coincides with the classical definition (see for example \cite{Cla,Ogg,Ogg1}) via \Cref{lem:YKXK}.

  Let $\fa$ be  an ideal of $\cO_F$ coprime to  all $v\in S_B$.  Let $S_\fa=\{v:v|\fa\}$.
For an ideal $\fN$ of $\cO_F$ coprime to $\fa$ and all $v\in S_B$, consider   the Atkin--Lehner quotients  
$$
Q^B_0(\fN)_{\fa}:= Y^B_0( \fa\fN)/\pair{w_v:v\in S_B\coprod S_\fa}.
$$  
Note that $Q^B_0(\fN)_{\fa}=X_K$  for $K\subset PB^\times(\BA_\rf)$  as follows: 
        $           K_{v}= {B_{v }}^\times$ for $v\in S_B,
$  $K_{v}=F_v^\times \Gamma_0\lb\fN\cO_{F_v}\rb <PB^\times_v=\PGL_2(F_v)$ for $ v\not
      \in S_B\cup S_\fa
$ and   finally
             \begin{equation*} K_{v}=F_v^\times \pair{\Gamma_0\lb\fN\cO_{F_v}\rb, g_v}\text{ for }   v\in
   S_\fa     \end{equation*}
      where $g_v$ is as in \eqref{ALA1} with $ m=\ord_v\fa$. 
       In particular, if $\ord_v\fa=1$, $K_v$ is the ramified maximal compact subgroup of   $PB^\times_v=\PGL_2(F_v)$, one of the only two 
maximal compact subgroups of   $PB^\times_v$,  with the other being $F_v^\times\GL_2(\cO_{F_v})$, see \eqref{cK}.
   Summarize  the  above discussion in the following lemma.
\begin{lem}\label{Kmax}
 For a maximal compact-modulo-center  subgroup $K$ of  $B^\times(\BA_\rf)$,  up to conjugation in $B^\times(\BA_\rf)$,  
  $X_K=Q^B_0(\cO_F)_{\fa}$ for some $\fa$ square-free.
   \end{lem}

 Now we consider gonality and genus of the Atkin--Lehner quotients.       
Let $Q^o$ be a connected component  of $Q^B_0(\cO_F)_{\fa,\BC}$. It is  the image  of a connected component  $Y^o$ of $ Y^B_0(\cO_F)_{\fa,\BC}$.  
Then  by Lemma \ref{lem:degXY} and Lemma \ref{lem:degX1X}, we have
\begin{align*} 
\gn(Q^o)\geq\frac{\gn(Y^o)}{2^{|S_B|+|S_\fN|}}\geq \frac{ 1}{| \cO^{\times,+} /\lb  \cO^\times\rb^2||\Cl|}\frac{\gn(X^1)}{2^{|S_B|+|S_\fN|}},
   \end{align*}
   where $X^1$ is defined   above Lemma \ref{lem:degXY} (with $\fN$ replaced by $\fa$).

\begin{prop} \label{lem:Qo1}
Let $F$ be a totally real number field. For $M>0$, the set of pairs $(B,\fa)$ as above such that 
$  {\gn(Q^o)} <M$  is contained in a  computable finite set.
 \end{prop}
 \begin{proof}
Let $\lambda$ be as above Proof of \Cref{thm:gn}.
By the argument in \cite{Abr}, we have 
\begin{equation}\label{Xgamma}
  \gn(X_\Gamma)\geq \frac{\lambda}{2}(g(X_\Gamma)-1),
\end{equation}
 where $g(X_\Gamma)$ is the genus of $X_\Gamma$.
  In particular, for any choice of $F,B,\fa$, we have
  $$
  \gn(X^1)\geq \frac{\lambda}{2}(g(X^1)-1).
  $$
 
  Zograf proved a lower bound on $g(X^1)$ in terms of  Laplacian eigenvalue and the volume of 
$  X^1$ which was computed by Shimizu, so that we have \cite[(3)]{Voi}  
 $$g(X^1)\geq \lambda \frac{D_F^{3/2}}{(2\pi)^{2[F:\BQ]}}\zeta_F(2)\Phi(B)\Psi(\fa)-1.$$
 Here   $D_F$   denotes the absolute discriminant  of $F$, $\zeta_F$ is the Dedekind zeta function of $F$,
 and
 $$\Phi(B)=\prod_{\fp\in S_B}\lb \Nm(\fp)-1\rb,$$
 $$\Psi(\fa)=\Nm(\fa)\prod_{\fp\in S_\fa}\lb 1+1/\Nm(\fp)\rb$$
 where $\Nm(\fa)=|\cO/\fa|$ is the ideal norm.   
 
 Combining the above inequities, we have
 \begin{align*} 
\gn(Q^o)\geq  c_1 \frac{  c_2 \Phi(B)\Psi(\fa)-2}{2^{|S_B|+|S_\fa|}},
   \end{align*}
   where $c_1,c_2$  are effective positive constants only depending on $F$.
As 
   \begin{align*} 
  \frac{\Phi(B)}{2^{|S_B|}}   \geq \frac{1}{ \prod_{\fp|2, \fp\in S_B}2} \frac{1}{ \prod_{\fp|3, \fp\in S_B}1} \frac{\prod_{\fp\nmid 6, \fp\in S_B}4}{ \prod_{\fp\nmid 6, \fp\in S_B}2} 
 \geq  \frac{1 }{2^{[F:\BQ]}} 2 ^{|S_B|-2[F:\BQ]}
   \end{align*}
  and 
   \begin{align*} 
  \frac{\Psi(\fa)}{2^{|S_\fa|}}   \geq    \frac{3 }{2}^{|S_\fa|},
   \end{align*} 
  for $M>0$, if  $\gn(Q^o)<M$, the possible cardinalities   $|S_B|, |S_\fa|$ 
   lie in a computable range.
  Similarly, one can deduce that  the possible norms of  elements in $S_B ,S_\fa$  lie in a computable range.
 The  lemma is  proved.
 \end{proof}

\begin{rmk} \label {eqgon}
If there is an explicit Sturm bound to determine a HIlbert modular form by a finite subset of all Fourier coefficient,
then the equation  $Q^o$ is computable  when $B$ is split at all finite places. Thus the gonality of  $Q^o$ is computable.
So the set of pairs $(B,\fa)$   such that 
$  {\gn(Q^o)} <M$ is actually computable.
 \end{rmk}
      \begin{prop}\label{lem:Bap}  
The set of quadruplets  $\{(F,B,\fa,\fN:g\lb Q^B_0(\fN)_{\fa}\rb= 0\}$ (here $\fa$ is not necessarily square-free) 
 is   a  computable finite set.

  \end{prop}  
     \begin{proof} 
      If $g\lb Q^B_0(\fN)_{\fa}\rb= 0$, then $g\lb Q^B_0(\cO_F)_{\fa\fN}\rb= 0$
  as it is dominated by $ Q^B_0(\fR)_{\fa}$. 
       By \Cref{lem:Qo1}, given $F$, a set  of  possible pairs $(B,\fb)$ such  that  $g\lb Q^B_0(\cO_F)_{\fb}\rb= 0$           is a  computable finite set.
A  set of  possible $F$ can be   determined from the the proof of
  \Cref{thm:gn}\footnote{And the algorithm is relatively more efficient compared to the whole   \Cref{thm:gn}.}.   
Since the genus of $Q^B_0(\fN)_{\fa}$ is computable, the final set in the lemma is computable.
  \end{proof}   
 
A priori, the proof of \Cref{lem:Bap} gives an algorithm to compute  this set  $\{(B,\fa,\fp:g\lb Q^B_0(\fp)_{\fa}\rb= 0\}$. However, it may be too large for   computation by hand. 
Instead, we  can take advantage of  the  classification of $(B,\fa)$ such that $g\lb Q^B_0(1)_{\fa}\rb=0$ (if available).  This classification is useful by   \Cref{lem:AL3} below. 
Let $$\fD=\prod_{p\in S_B}\fp,$$  which is the discriminant of $B$ and determines $B$ by the Hasse principle  for quaternion algebras.

       \begin{lem}\label{lem:AL3}

  \begin{enumerate}
  Consider the following conditions:
\item 
  $g\lb Q^B_0(\fN)_{\fa}\rb=0$,
  \item
   $g\lb Q^B_0(1)_{\fa}\rb=0$,  
   \item
   $g\lb Q^B_0(1)_{\fa\fN}\rb=0$,
    \item
    $g\lb Q^B_0(\fN)_{\fb}\rb=0$ for every    $\fb$ such that $\fb |\fa$ and $\fb\neq \fa$,
    
   \item   there is no Hilbert newform   of level $\fD\fN\fa$ with trivial central character whose Atkin--Lehner sign is $-1$ at every prime $\fp|\fD$, and $1$ 
  at every prime $\fp|\fa$.

  \end{enumerate} 
 
  We have (1)$\Rightarrow$(2)(3)(4)(5) (here $\fa$ is  not necessarily square-free). If $\fN$ is a prime and $\fa$ is square-free, then (2)+(4)+(5)$\Rightarrow$(1).

  \end{lem}

   \begin{proof} 
First, by the natural dominant morphisms $ Q^B_0(\fN)_{\fa}\to  Q^B_0(1)_{\fa} $, 
 and 
   $Q^B_0(\fN)_{\fa}\to Q^B_0(1)_{\fa\fN}$, (1)$\Rightarrow$(2)(3).
   
 Second, let us show that  (1)$\Rightarrow$(4). Let $S$ be the set of primes $\fp$ such that $\ord_{\fp}\fa=\ord_{\fp}\fb$, we prove an even stronger statement: the genus of 
  $  Y_0(\fN\fb)/\pair{w_{\fp}:p\in S}$ is 0. Assume that statement is wrong. Then there is
  a representation  $\Pi$  appearing in $Y_0(\fN\fb)$ such that  the Atkin--Lehner sign  of $\Pi_{\fp},\fp\in S,$ is $1$. 
 Note that $\Pi$ also appears in $Y_0(\fN\fa)$.
For $p\not\in S$, by \Cref {levelup}, the $1$-eigenspace of $w_{\fp}$ on $\Pi_{\fp}^{\Gamma \lb \fN\BZ_{\fp}\rb} $ is nonzero. 
So $g(Q^B_0{\fN }_{\fa})\neq0$, a contradiction.

Third,   (1)$\Rightarrow$(5). Indeed, if there is such a newform as in (5),  the Jacquet--Langlands correspondence to $B^\times$ of its corresponding automorphic representation appears in $Q^B_0(\fN)_{\fa}$ (see \Cref{minimal} (2) and \Cref{JLpm1}), a contradiction.

Finally, let us consider the other direction (2)+(4)+(5)$\Rightarrow$(1).
Assume there is  a representation  $\Pi$  appearing in $Q^B_0(\fN)_{\fa}$ of level $\fc$. Then $\fD|\fc$ and   $\fc|\fD\fN\fa$.
  So by (5), $\fc \neq \fD\fN\fa$ (see \Cref{minimal} (2) and \Cref{JLpm1}). By (2), $\fc\nmid \fD\fa$.  
  As $\fN$ is a prime and $\fc \neq \fD\fN\fa$,
  $\fc=\fD\fN\fb$ for some   $\fb$ such that $\fb |\fa$ and $\fb\neq \fa$.  
 For every prime $\fp|\fb$, since $\ord_{\fp}\fb=\ord_{\fp}\fa$ (as $\fa$ is square-free),  the Atkin--Lehner sign  of $\Pi_{\fp}$ is $1$.
 Thus   $\Pi$  appears in $Q^B_0(\fN)_{\fb}$.
 This is impossible by (4). 
   \end{proof}
  
    Let
$F=\BQ$. 
Let us compute this set $\{(B,\fa,\fp:g\lb Q^B_0(\fp)_{\fa}\rb= 0\}$.
For an ideal, use its positive generator to denote it instead. Let $\fD=(D),\fp=(p)  ,\fa=(a))$ where $D,p,a>0$.
  Then $D,p,a$ are coprime to each other.

\begin{prop} \label{prop:ALg=0}   
 
Assume that $a$ is square-free and $D>1$. Then   the genus $g(Q^B_0(p )_a)=0$ for a prime number $p$ if and only if $ a=1$ and one of the following  holds:

$$D=6, \quad p=5,7, 13;
$$
$$D=10, \quad p=3,7; 
$$
$$D=14,\quad  p=3,5;  
$$
$$(D,p)= (15,2),
      (21,2),
           (26,3),
      (35,2),
      (39,2).
$$

\end{prop}
\begin{proof} We first consider the case $a=1$. 
By \Cref{lem:AL3}, $g(Q^B_0(p )_a)=0$  then implies that $g(Q^B_0(1 )_p)=0$  and $g(Q^B_0(1 )_1)=0$. 
 Such $(D,p)$'s  are classified in  \cite[Proposition 4.1]{Rie}. 
  Using \textsf{LMFDB}, we find that only $(D,p)$'s as in the proposition satisfy \Cref{lem:AL3} (5) 
   The proposition follows from \Cref{lem:AL3}.

  In general, by  \Cref{lem:AL3}, $g(Q^B_0(p )_a)=0$,
  implies $g(Q^B_0(p )_1)=0$.  So we only need to check $(D,p)$'s as in the proposition. 
 By  \Cref{lem:AL3}, it also  implies   $g(Q^B_0(1 )_{pa})=0$.
  Now the  proposition  follows from  the  classification of $(B,\fa)$ such that $g\lb Q^B_0(1)_{\fa}\rb=0$, which is given in \cite[Proposition 4.1]{Rie}).\end{proof}  

  The analog of the proposition for $D=1$ is not needed.

   \subsection{Auto-critical  curves}\label{Good  curves}

An    irreducible infinite dimensional  representation  $\Pi$
of 
$B^\times(\BA_\rf)$ is called automorphic and  holomorphic of weight $2$  if 
its Jacquet--Langlands correspondence   to $\GL_{2,F}$ is the finite part of a   cuspidal automorphic representation  holomorphic of weight $2$. 
We say that $\Pi$ \textit{appears}  in $X_K$ if the subspace of $K$-invariants $\Pi^K\neq \emptyset$. 
Then by \cite[Theorem 3.7]{YZZ},
\begin{equation}\label{eq:H10}
 H^{1,0}(X_{K,\BC})\cong \bigoplus\Pi^K.
\end{equation}
 the multiplicity-free    direct sum over all   representations 
 of $B^\times(\BA_\rf)$     \textit{appearing} in $X_K$. %, and the isomorphism is Hecke equivariant (which is not needed for us and thus not elaborated).

We say that the Shimura curve $X_K$  is auto-critical if    for any  triple $(\Pi_1,\Pi_2,\Pi_3) $ of 
  irreducible   representations  of 
$B^\times(\BA_\rf)$  \textit{appearing in} $X_K$,  the space of $B^\times(\BA_\rf)$-invariant linear forms
 $$  \Hom_{B^\times(\BA_\rf)}\lb \Pi_{1}\otimes \Pi_{2}\otimes \Pi_{3},\BC\rb= 0.$$ 
 In this case, we also say that a connected component of $X_{K,\BC}$ is auto-critical.
 Let $\Pi_{K}$ be the direct sum of 
 irreducible   representations  of 
$B^\times(\BA_\rf)$   appearing in $X_K$. Then 
$X_K$ is auto-critical if and only if
 $  \Pi^{\otimes 3}$   has no nonzero  $B^\times(\BA_\rf)$-invariant linear forms.
 This is the formulation in the Introduction.
  
  The notion ``auto-critical Shimura curve" is defined regardless of the genus. In particular, a Shimura curve of genus 0 is always auto-critical.

The notion ``auto-critical Shimura curve" is useful  by the following.
   \begin{lem}[{\cite[Corollary 3.3.4]{QZ2}}]\label{vancor2} 
Auto-critical Shimura curves of genus at least 2 are critical. 
\end{lem} 
\begin{rmk}\label{weakauto-criticalnes}
  Actually, we may weaken the definition of auto-criticalness by only requiring  $B^\times(\BA_\rf)$-invariant linear forms to vanish on relevant test vectors. 
  In this case, the Shimura curves are still critical. 
  This shall be studied in a subsequent work.  
   \end{rmk}
 
    We reformulate \Cref{thm:finiteintro} as follows.
   \begin{thm} \label{thm:finite}
    There are only finitely many  almost definite   quaternion algebras   $B$ over all totally real number fields,  
     such that the Shimura curve $X_K$  is auto-critical  for some open compact-modulo-center subgroup $K \subset B^\times(\BA_\rf)$.
Moreover, this finite set is contained in a  computable 
 set of  explicit  almost-definite quaternion algebras over totally real fields.
    \end{thm} 
      \Cref{thm:finite}    will be proved in the end of this section.

   By definition, if representations appearing in  two  Shimura curves  are comparable, we may deduce the auto-criticalness of one  from another.
 The following obvious lemma will be very helpful later.
 \begin{lem}\label{lem:auto-critical} 
A Shimura curve   that is dominated by a auto-critical Shimura curve under a morphism between Shimura curves is itself auto-critical.
\end{lem}

Let us also define \textit{subhyperelliptic} curves to be  curves of gonality at most 2, that is,  admitting a degree 2 morphism to $\BP^1$. This is a more flexible notion than hyperellipticity for our convenience. 
Then \Cref{lem:auto-critical} is an analog of the fact that a curve dominated by a subhyperelliptic curve is itself subhyperelliptic (as well as 
the fact that a curve dominated by a  critical curve  
   itself  is critical   \cite[Proposition 2.3.3]{QZ1}).
  The sets of subhyperelliptic and auto-critical curves do not include each other.
  See \Cref{thm:auto-critical} (1). 
However, sometimes subhyperellipticity could be useful to check auto-criticalness, as we now discuss.

Let $\Pi$ be an  irreducible   representation   of $B^\times(\BA_\rf)$ 
   appearing in $Y_0^\fA(\fN)$. Then for $v\in S_B$, $\Pi_v=\pm1$ of $B_v^\times/F_{v}^\times K_v \cong \BZ/2$.
For
    $v\nmid \fN$, $\Pi_v$  is an unramified principal series of $\GL_2(F_v)$
   or $\St\otimes \pm 1$ (the Steinberg representation $\St$  of  $\GL_2(F_v)$ and  its   unique unramified quadratic twist $\St\otimes -1$, see  \ref
    {Notions}).    
    
  \begin{lem}
\label {lem:hypauto-critical}
Let $S $ be a finite set of finite places $v$ of $F$ such that   $\ord_v\fA\fN=1$.
If   $g\lb Y_0^\fA(\fN)/\pair{w}\rb=0$, where    $w=\prod_{v\in S} w_v$, then $Y_0^\fA(\fN)$ is auto-critical.
 \end{lem}
  \begin{proof}
 Claim: for $\Pi$ appearing in $Y_0^\fA(\fN)=X_K/\BA_F^\times$, $\Pi_v=\St\otimes \pm 1$ for $v\in S\bsl S_B$. Indeed, otherwise, $\Pi_v$  is  unramified,
 so that by \Cref {levelup}, the $\pm1$-eigenspace of $w_v$ on $\Pi_v^{K_v}=\Pi_v^{\Gamma_0 \lb \fN\cO_{F_v}\rb} $ is nonzero.  
Thus the involution $w$ acting on
$\Pi^K$ has both $\pm1$-eigenvalues.  So $g\lb Y_0^\fA(\fN)/\pair{w}\rb\neq 0$,  a contradiction.

 As $\Pi_{i,v}=\pm 1$ for $v\in S\cap S_B$, 
by the claim, $\dim \Pi_v^{K_v}=1$ for $v\in S$. 
Let $e(\Pi_v^{K_v})$ be of the (eigen)value  of $w_v$   on the $1$-dimensional $\Pi_v^{K_v}$. 
Then $g\lb Y_0^\fA(\fN)/\pair{w}\rb=0$ which implies that  
$$-1=e(\Pi):=\prod_{v\in S}e(\Pi_v^{K_v}).$$  So for $\Pi_i,i=1,2,3,$ appearing in $Y_0(\fN)$,   we have $e(\Pi_1)e(\Pi_2)e(\Pi_3)=-1$. Then for some $v_0\in S$, we have 
$$-1= \prod_{i=1}^3e(\Pi_{i,v}^{K_v}).$$ 

If $v_0\in S_B$ so that $\Pi_{i,v}=\pm 1$, then clearly, the component at $v_0$ of
 $  \Pi_1\otimes \Pi_2\otimes \Pi_3$,  and thus   $  \Pi_1\otimes \Pi_2\otimes \Pi_3$,  has no nonzero  $B^\times(\BA_\rf)$-invariant linear forms.
  If $v_0\not\in S_B$, apply
\Cref{lem:ep}    (2). 
 \end{proof}

\begin{rmk}
 Without the condition that   $\ord_v\fA\fN=1$, $g\lb Y_0^\fA(\fN)/\pair{w}\rb=0$ still implies that  
$ Y_0^\fA(\fN)$ is critical if its geometrically connected components have genus at least 2, by hyperellipticity as we discussed in the Introduction. 
Such criticalness is actually  explained by \Cref{weakauto-criticalnes}.
 
   \end{rmk}
   
%(2)  The case $N=49$ of \Cref{thm:auto-critical} (1) reflects an observation of   Gross and Prasad   on test vectors  for trilinear forms \cite[Remark 7.6]{GP}.

Now we prepare to prove \Cref{thm:finite}.    If $\fa$ is square-free and $\Pi $  appears in $Q^B_0(\fN)_{\fa}$, then $\Pi_{v}=1$  for $v\in S_B$, and 
  $\Pi_v$ is a(n unramified) principal series or $\St\otimes - 1$  (see \Cref{minimal} (2))  for $v\in S_\fa$.
  By   \Cref {lem:ep} (1) (2),  
        for any  triple $(\Pi_1,\Pi_2,\Pi_3) $ appearing in $C$,  we  have
         \begin{equation}\label{pNShAL}
   \Hom_{B_v^\times}\lb \Pi_{1,v}\otimes \Pi_{2,v}\otimes \Pi_{3,v},\BC\rb\neq 0
, \text{ for }v\nmid \fN.
 \end{equation} 
  In other words,  we can treat $S_B\cup S_\fa$ as a set of ``unramified" places.   
  By \Cref{lem:auto-critical} and \eqref{pNShAL}, we also have the following.
  \begin{lem}\label{lem:AL} 

   If  $\fa$ is square-free and $Q^B_0(\fN)_{\fa}$ is   auto-critical, then $ Q^B_0(\cO_F)_{\fa} $ has genus $0$.

       \end{lem}  
        Combining the last  lemma and \Cref{lem:Bap}, we have the following.
      \begin{prop} \label{thm:ALfinite}
 Consider the set of triples $(F,B,\fa)$   with   $\fa$ being square-free   such that  $Q^B_0(\fN)_{\fa}$ is auto-critical for some $\fN$. 
    This  set is     a  computable finite set.  $ need auto cric to be computable $algo GL2 actually.
       \end{prop}

  Now we can prove  \Cref{thm:finite}. 
  
     \begin{proof} [Proof of \Cref{thm:finite}]
By
  \Cref  {lem:auto-critical}, we only need to prove  \Cref{thm:finite} for   $Y_K$ with $K$ maximal.
   By  \Cref {Kmax},  \Cref{thm:finite} follows from \Cref{thm:ALfinite}.
                  \end{proof} 

 \section{Classification of auto-critical Shimura curves over $\BQ$} 
 
 We continue to use the notations in the last section. Besides, we add 
some more conventions.
For an ideal of $\BZ$, use its positive generator to denote it instead.   We want to classify auto-critical $X_*^A(N) =X_*^\fA(\fN)  $ with $\fA=(A),\fN = (N)$ and    $Y_*^A(N)$ for $*\in \{0,1,\emptyset\}$. Recall that $A$ is a product of all primes  in $S_B$ with positive powers.    If $A$ is  square-free, then
               $X_*^A(N)$ and    $Y_*^A(N)$  are  denoted by $X_* (N)$ and    $Y_* (N)$ to simplify the notations.  
                 For example,  if $B = \RM_{2,\BQ}$,      $ X_*(N)$ is a classical modular curve.

We first recall the necessary background on modular forms and automorphic representations relevant to determining the auto-criticalness of Shimura curves, as described in \Cref{newforms in the database LMFDB}. We also present examples of Shimura curves for which such computations are carried out.
In the remainder of this section, we classify the auto-critical Shimura curves over $\BQ$ with classical level structures, making use of these examples.

 \subsection{Algorithms and examples}\label{newforms in the database LMFDB}
We study irreducible infinite-dimensional representations $\Pi$ of $B^\times(\BA_\rf)$ appearing in Shimura curves over $\BQ$ via the Jacquet--Langlands correspondence to $\GL_{2,\BQ}$.
The cuspidal automorphic representations $\pi$ of $\GL_{2,\BQ}$ whose infinite component is the holomorphic discrete series of weight~2 are in bijection with newforms $f$ (in the sense of classical modular forms for congruence subgroups of $\SL_2(\BZ)$) of weight~2.
Let the level of $\pi$ be the level of $f$.
Then for a prime number $p$, the conductor of $\pi_p$ (see \Cref{Notions}) is exactly the $p$-exponent of the level.
Recall that $S_B$ is the finite set of finite places $v$ of $F$ such that $B_v := B(F_v)$ is a division quaternion algebra.
Then $\pi_\rf$ (the finite component of $\pi$) is the Jacquet--Langlands correspondence to $\GL_{2,\BQ}$ of an irreducible infinite-dimensional representation $\Pi$ of $B^\times(\BA_\rf)$ if and only if $\pi_v$ is a discrete series for every $v \in S_B$.
In this case, if $p \notin S_B$, then $\Pi_p = \pi_p$.
Let the level of $\Pi$ be the level of $\pi$, which is a reasonable notion by \eqref{JLcond}.
Note that every $p \in S_B$ divides this level.
The most important task is to determine whether a Shimura curve $X_0^A(N)$ is auto-critical by \Cref{lem:auto-critical}, since it is dominated by other curves $X_1^A(N)$ and $X^A(N)$.
Then $\Pi$ appears in $X_0^A(N)$ if and only if $\Pi$ has level dividing $AN$ (by \eqref{JLcond}).

\begin{prop}\label{globalalgo}
It is computable whether $X_0^A(N)$ is auto-critical.
\end{prop}

\begin{proof}
It is well known that the finite set of newforms of a given level $f$ is computable.  
By \cite{LW}, one can further compute the local components of all corresponding cuspidal automorphic representations $\pi$ at a given prime $p$.  
Since it is known a priori that $\Pi_p = \pi_p$ is a(n unramified) principal series if $p \nmid AN$,  
by \Cref{lem:ep}(1), we only need to compute 
\[
\Hom_{B_p^\times}\!\left( \Pi_{1,p} \otimes \Pi_{2,p} \otimes \Pi_{3,p}, \BC \right)
\]
for $p \mid AN$, where $\Pi_i$ are the representations appearing in $X_0^A(N)$.  
This is computable by \Cref{localalgo}.
\end{proof}

Given the progress on computing Hilbert newforms, \Cref{globalalgo} might be extended to Shimura curves over totally real fields.

The proof of \Cref{globalalgo} can be turned into an explicit algorithm using the computer algebra system \textsf{Sage} \cite{Sage}.  
Indeed, all steps in the proof have been implemented, except for \Cref{localalgo}.  
However, we find it more convenient to perform the computation ``by hand",  
using data from the \textsf{LMFDB} database \cite{LMFDB} of modular forms,  
and replacing \Cref{localalgo} with certain practical criteria from Appendix~\ref{Local root numbers} in most cases.  
This suffices to determine the \textit{finite set} of auto-critical Shimura curves of the forms $X_*^A(N)$ and $Y_*^A(N)$.  
We often start with $Y_0^A(N)$, which is dominated by $X_0^A(N)$.  
In this case, we simply require $\Pi$ and $\pi$ to have trivial central character.

  Now we  give a detailed description of our procedure to determine whether   $X_0^A(N)$ is auto-critical, mostly using \textsf{LMFDB} and ``by hand". 
  This enables us to compute some examples of Shimura curves that we will use later.  
  
  Let $\Pi_{i},i=1,2,3$ be  irreducible infinite dimensional  representations  of $B^\times(\BA_\rf)$  appearing in Shimura curves over $\BQ$. 
We compute $  \Hom_{B^\times(\BA_\rf)}\lb \Pi_{1}\otimes \Pi_{2}\otimes \Pi_{3},\BC\rb$. If it is nonzero,  $X_0^A(N)$ is not auto-critical and we are done.  In most examples we compute below, we can find such a triple easily.  
Otherwise, we  loop over all possible triple representations, and show that  $X_0^A(N)$ is  auto-critical

  We separate the computation of this $\Hom$ space
   into seven steps. 
  Let $\Pi$ be one of $\Pi_{i,p},i=1,2,3$. In the first four steps, we mainly determine $f, \pi$ such that $ \pi_\rf$  is 
the Jacquet--Langlands correspondence to $\GL_{2,
\BQ}$ of  $\Pi$ as above and 
$f$ is the corresponding newform.

First, 
 we call $f$ minimal if $\pi_p$ is minimal for all $p$. This is the equivalent to the  definition in \textsf{LMFDB}, and %classical definition of a minimal newform, by$noneed \Cref{minimal} (1) and the bijection between Dirichlet characters and finite order Hecke characters of $\BQ^\times$. $also cite LW AL Definition 2.7 ([AL78, p. 236]). The newform f is p-primitive if the p-part of its level is minimal among all its twists.
%Clearly, f is p-primitive if and only if ?f,p is a primitive representation of G}. 
whether or not $f$  is minimal is displayed in \textsf{LMFDB}.
If $f$ is minimal, use \Cref  {minimal} to conclude  if the corresponding  representation $\pi_p$ at $p$ is a principal series, a special representation $\St_{\BQ_p}\otimes \chi$  where $\chi$ is an unramified  character of $\BQ_p^\times$, or a minimal supercuspidal representation.
If $f$ is not minimal, a twist of $f$ that is minimal is also provided  in \textsf{LMFDB}. 
It turns out that the only non-minimal representations we will  use  are ramified twist of $\St_{\BQ_p}$, but not supercuspidal representations. (The principal series are  dealt in the following second step without the minimality information.)

Second,   if $\pi_p$   is a principal series and  (so that $B_p$ is a matrix algebra), we   apply  \Cref {lem:ep}  (1) to conclude that  
$  \Hom_{B_p^\times}\lb \Pi_{1,p}\otimes \Pi_{2,p}\otimes \Pi_{3,p},\BC\rb\neq 0$

Third,
  if $\pi_p=\St_{\BQ_p}\otimes \chi$  where $\chi$ is an unramified,  by  \cite[Proposition 2.8]{LW},
\begin{equation}
\label{chiap}
\chi(p)=a_p(f),
\end{equation}
 the $p$-the Fourier coefficient of $f$.  This determines $\chi$. 
Moreover,  when $f$ has trivial central character, \textsf{LMFDB} also displays Atkin--Lehner sign of $f$ at $p$, which is  $-a_p(f)$ (see \Cref{minimal} (2)).

Fourth, if $\pi_p$ is minimal supercuspidal,    we  can use the computer algebra system  \textsf{Sage} \cite{Sage} to find the compact induction data (thanks to the algorithm in  \cite{LW}, and note the ``dual" in \cite[Theorem 4.6]{LW}), see \Cref{thm:type}.  However, this is 
only needed in a few examples. In most cases, conductors are enough. See the sixth and seventh steps below.
   
Fifth,  now we   start to discuss the computation of   $  \Hom_{B_p^\times}\lb \Pi_{1,p}\otimes \Pi_{2,p}\otimes \Pi_{3,p},\BC\rb$ without a principal series among  the  Jacquet--Langlands correspondence  of $\Pi_{i,p},i=1,2,3$ to $\GL_2$. 
If  the product of the  central  characters of $\Pi_{1,p},  \Pi_{2,p},  \Pi_{3,p}$ 
is not trivial, then    this space is apparently $0$. See \Cref
 {lem:cc}. (Actually,  most cases we compute below concern $Y_*^A(N)$, so that  all of $\Pi_{1,p},  \Pi_{2,p},  \Pi_{3,p}$ 
 have  trivial central  characters. Then this step is redundant.)

 Sixth, assume $B_p$ is split.  We have some convenient  criteria to apply. 
First, if two of $\Pi_{i,p},i=1,2,3$ take the form $\St_{\BQ_p}\otimes \chi$ (not necessarily minimal) of $\GL_{2}(\BQ_p)$, we  can apply \Cref {lem:ep}  (2).  
Second,   if   one of $\Pi_{i,p},i=1,2,3$ takes the form $\St_{\BQ_p}$, 
we  can       apply \Cref {lem:ep}  (3).
Third, if 
$\Pi_{i,p},i=1,2,3$  are discrete series of $\GL_{2}(\BQ_p)$ and 
 of  conductors  $\Cond_1\leq \Cond_2<\Cond_3$, we  can     apply \Cref {lem:ep} (4).  
     Fourth, in a few cases, we can transfer the problem of computing 
  $  \Hom_{B_p^\times}\lb \Pi_{1,p}\otimes \Pi_{2,p}\otimes \Pi_{3,p},\BC\rb $
 to the division quaternion algebra over $\BQ_p$ in the seventh step below. For  demonstrations,  see Example \ref
  {eg:6egs}  (1), and more generally  \Cref{cor:Dn} and \Cref{cor:Dn1}.  
    These  four cases will cover most of the examples below. 
 In  the rest of examples below,  we will being using \textsf{Sage}. There, $\Pi_{i,p} $'s will  all be minimal supercuspidal representations of of $\GL_{2}(\BQ_p)$ of the same conductor, we first find  the compact induction data using \textsf{Sage}, then apply \Cref{scusp}.  For a demonstration, see Example \ref
  {eg:6egs}  (2).

 Seventh,  assume $B_p$ is division.   Then $\Pi_p$ is essentially a representation of a finite quotient of $B_p^\times$. Similar to the sixth step, we have some criteria  
 in \Cref{Quaternion algebra} that will be enough to cover most example.  In the remaining few example,  we use \textsf{Sage} and character identities \eqref{character identity2} and \eqref{character identity1} to obtain information on $\Pi_{i,p} $'s. Then we work with the character table of the finite quotient of $B_p^\times$ to compute $  \Hom_{B_p^\times}\lb \Pi_{1,p}\otimes \Pi_{2,p}\otimes \Pi_{3,p},\BC\rb$.

 This finishes the  algorithm part of this subsection.
 % Another way of computation is using \Cref{thm:Pra} (2) on trilinear forms and root numbers. However, the computation of root numbers at $2$ seems too complicated.

    Now   we list some examples of Shimura curves  that we succeed in determining their auto-criticalness and  we will use later.  The reader may skip them for now and come back for reference.
       The examples are arranged in a such way  that is handy for
     being used later and also convenient for the proofs in  \cite{Qiufinite}.
    The  computations  follow the above steps, and will appear in a separate document \cite{Qiufinite} (except for the demonstrative  Example \ref
  {eg:6egs}). In \cite{Qiufinite}, we also give more criteria for computing trilinear forms.
 
   Let $F=\BQ$, $B=\RM_{2,\BQ}$ and $C= X_0(N)=Y_0(N)$ until Example \ref{eg:2,5}.
       \begin{eg}  \label{eg1} 
      
  (1)    For $N=p^r= 2^6, 3^4,5^4 ,7^3$,  $C$ is not auto-critical.    
    
(2)  For $p^r=13^2$, $C$ is not auto-critical.       \end{eg} 
 
      \begin{eg}  \label{eg2}

 (1) For $N=p\cdot q=5\cdot13,7\cdot13$, $C$ is not auto-critical.  
 
  (2) For $N=p^2\cdot 13, p=2,3$, $C$ is not auto-critical. 
        
      (3) For $N=2\cdot 3 \cdot 13$, $C$ is not auto-critical.

      \end{eg}

    \begin{eg}  \label{eg3}

    (1)    
 For $N=p\cdot 7^2, p=2,3,5$, $C$ is not auto-critical.  
 
   (2)      For $N=2^3\cdot 7,3^2\cdot 7,5^2\cdot 7$, $C$ is not auto-critical.
   
       (3)    
 For $N=p\cdot q \cdot 7$ where $p\neq q\in \{2,3,5\}$ are distinct, $C$ is not auto-critical.      \end{eg} 
       
          \begin{eg}  \label{eg4}

    (1)    For $N=2^2\cdot 5^2 $, $C$ is not auto-critical.  
      
   (2)    
 For $N=3\cdot 5^2 $, $C$ is not auto-critical. 
 
 (3)   For $N= 5^3 $, $C$ is not auto-critical. 

        \end{eg}

     \begin{eg}  \label{eg5}

    (1)     For $N= 3^2\cdot 5$, $C$ is not auto-critical. 
    
    (2)      
 For $N= 2^4\cdot 5$, $C$ is not auto-critical.   
  
  (3)    
 For $N=2^2\cdot 3 \cdot 5$, $C$ is not auto-critical.
   \end{eg}

   \begin{eg}  \label{eg6}

 (1)      For $N=2^2\cdot 3^3$,  $C$ is not auto-critical.   
    
 (2)     For $N=2^4\cdot 3^2$,  $C$ is not auto-critical.

(3)  For $N=2^5 \cdot 3 $, $C$ is not auto-critical. 
     \end{eg} 
      
   \begin{eg}\label{eg:6egs}     
    Let  
 $N=  20, 24, 
27,   32, 
36,    
49.$ From \textsf{LMFDB},  we find that for 
$$
(N,p)=(20,2), (24,2), (32,2),(36,2), (27,3),(49,7),
$$
  there is only one $\Pi$ appearing  in $C$ and   $\Pi_p$ is minimal.  
  
  (1) For $(N,p)=(20,2), (24,2), (36,2),  \text{ (resp. }    (49,7) $), we  use \Cref{cor:Dn1} (resp. \Cref{cor:Dn} (2)) to conclude that $C$ is auto-critical (resp. not auto-critical). 
  
  (2) 
For $(N,p)=    (32,2),(27,3)$,  since $\Pi_p$ is minimal supercuspidal and  thus is the compact induction of  a   very cuspidal representation  $M$ of $\cK^{\ram}$, the ramified 
maximal compact subgroups of   $\PGL_2(\BQ_p)$
 (see \Cref{thm:type}), we   use the computer algebra system  \textsf{Sage} \cite{Sage} to find the compact induction data   and check   the trilinear form spaces in
 \Cref {scusp}.  
 We   find that  $\dim M=2$ and $ \begin{bmatrix}
0 & -1 \\
p & 0
\end{bmatrix}\in \cK^{\ram}$ acts on $M$ as $-1$. Thus
 $M^{\otimes 3}$  does not   have  nonzero $\cK^{\ram}$-invariant linear forms.  By  \Cref {scusp}, $C$ is auto-critical.  
    
\end{eg}

  \begin{eg}\label{eg:5egs}
 For $N=40,48,50,54,72$, $C $  is auto-critical.  
 \end{eg}

  \begin{eg}\label{eg:2,5}
 
    The modular curve    $X(2,5)$ is auto-critical.
         \end{eg}
     
    Now we consider  $F=\BQ$ and $B\neq \RM_{2,\BQ}$.

      \begin{eg}\label{eg:ShA1} 
        Let $S_B=\{2,3\}$ and $C=Y_0^A(1)$.

(1)  For $A=2\cdot 3^4$, $C$ is not auto-critical. 
      
      (2) For $A=2^2\cdot3^3$, $C$ is not auto-critical.

  (3) For $A=2^5\cdot3$, $C$ is not auto-critical.  
  
(4) For $A=2\cdot 3^3=54$, $C$ is  auto-critical.

(5) For $A=2^4\cdot 3^2=144$, $C$ is  auto-critical.

      \end{eg}

 \begin{eg}\label{eg:ShA2} 
        Let $S_B=\{2,5\}$ and $C=Y_0^A(1)$.

(1)  For $A=2\cdot 5^3$, $C$ is not auto-critical.

    (2)   
     For $A=2^3\cdot 5=40$, $C$ is not auto-critical. 
           
        (3) For $A=2^2\cdot5^2$, $C$ is  auto-critical.

      \end{eg}  
      
        \begin{eg}\label{eg:ShA3} 
        Let $S_B=\{2,11\}$ and $C=Y_0^A(1)$.

(1)   For $A=2\cdot11^2=242$, $C$ is  not auto-critical.   
      
  (2)  For $A=2^5\cdot 11=352$, $C$ is not auto-critical. 
   
(3) For $A=2^4\cdot 11$, $C$ is  auto-critical.

       \end{eg} 

      \begin{eg}\label{eg:ShA4} 
        Let $A=11^n\cdot 3$ and $C=Y_0^A(1)$.

   For $A=11^2\cdot 3$, $C$ is not auto-critical.

          \end{eg}

      \begin{eg}\label{eg:ShA9} 
     
        Let $A=2^n\cdot q,q=7,17,29,41$ and $C=Y_0^A(1)$.

(1)   For $ A= 2^3\cdot q$, $C$ is   not auto-critical.   

(2) For $ A= 2^2\cdot q$, $C$ is  auto-critical.    
 \end{eg}

    \begin{eg}\label{eg:ShA7} 
        Let $A=5^n\cdot 3$ and $C=Y_0^A(1)$.

(1)   For $ A= 5^3\cdot 3$, $C$ is   not auto-critical.

(2) For $ A= 5^2\cdot 3$, $C$ is  auto-critical.
       \end{eg} 
   
         \begin{eg}\label{eg:ShA8} 
      Let $A=3^n\cdot q, q=5,7,19,31$ and $C=Y_0^A(1)$.

(1)   For $ A= 3^3\cdot q$, $C$ is   not auto-critical.  

(2)  For $ A= 3^2\cdot 5,$ $C$ is auto-critical.

(3) For $ A= 3^2\cdot q,q=7,19,31$, $C$ is  not auto-critical.

  \end{eg}

         \begin{eg}\label{eg:ShA6} 
          Let $A=7^n\cdot 2$ and $C=Y_0^A(1)$.  
          
(1)   For $ A= 7^3\cdot 2$, $C$ is   not auto-critical.

(2) For $ A= 7^2\cdot 2$, $C$ is  auto-critical.

         \end{eg}

             \begin{eg}\label{eg:ShA5}

    For $A=  23^2\cdot 2$, $C=Y_0^A(1)$ is not auto-critical.

       \end{eg}

  \begin{eg}\label{eg:ShA10} 
    Let $A=p^nq$ and $C=Y_0^A(N)$.

(1)   For $p=2,q=3$, $C$ is not auto-critical in any of the following cases:
 $$n=2,\quad N=19,43;$$
 $$ n=3,\quad N=5,7,13,19.$$

  (2) For  $A=3^2\cdot 2,N=17$,  $C$ is not auto-critical.
  
    (3) For  $A=3^3\cdot 2,N=5,13$,  $C$ is not auto-critical.

 (4) For $A=2^2\cdot 3, N=5,7,13$, $C$ is  
  auto-critical.
  
(5) For  $A=3^2\cdot 2,N=5,13$,  $C$ is  auto-critical.   
 
 (6) 
  For $ A= 3^2\cdot 7,N=2$, $C$ is    auto-critical.

      \end{eg}
      
        \begin{eg}\label{eg:ShA11} 
          Let $A=p^nq$ and $C=Y_0^A(N)$.
          
   (1) For $n=2$,   $C$ is not auto-critical  in any of the following cases:
  $$p=2,q=5,  \quad N= 7,13;$$
$$(p,q,N)=(2, 11, 3),(2,17,3),(5,3,7).$$ 

  (2) For $A=2^2\cdot 5,N= 3$, $C$ is  auto-critical.

       (3) For $A=5^2\cdot 2,N= 3,19$, $C$ is not auto-critical.

 \end{eg}  
 
 Below, we consider more general levels.
   \begin{eg}\label{eg:ShAX0} 
    Let   $C=X_0^A(1)$.

   (1) For $A=5^2\cdot 2$, $5^2\cdot 3$, $7^2\cdot 2$, $C$ is   auto-critical. 
  
  (2) For $A=2^4\cdot3^ 2$, $C$ is  not  auto-critical.    
   \end{eg}
\begin{eg}\label{eg:ShAX1} 
    Let   $C=X_1^A(N)$ where $A=p^nq$ and $N$ is a prime.

   (1) For $A=2^2\cdot 3$, $N=5$, $C$ is not  auto-critical.

    (2) For $A=2^2\cdot 3$, $N=7$, $C$ is not  auto-critical.

  (3) For $A=3^2\cdot 2$, $N=5$, $C$ is not  auto-critical.

     \end{eg}

   \subsection{Classical modular curves}
\label{Standard odular curves}

       Let $F=\BQ$, $B=\RM_{2,\BQ}$ so that the Shimura  curves  are classical  modular curves. 
       We first recall  the classification of subhyperelliptic   modular curves $X_0(N)$.
       As they have cusps which are rational points over $F$, they are  subhyperelliptic   if and only if their base changes to $\BC$ are  subhyperelliptic.
  
  The modular curve
 $ X_0(N)$  has genus $  0$ if the only if 
$N$ is in the following list:
 \begin{equation}\label{g=0}
1, 
2, 
3, 
4, 
5, 
6, 
7, 
8, 
9, 
10,  
12, 
13,  
16,  
18,  
25, 
 \end{equation} The prime numbers are
$2,3,5,7,13$. 
And   $ X_0(N)$  has genus $  1$ if the only if 
$N$ is in the following list:
 \begin{equation}\label{g=1}
 11, % 11
14, % 7
15, % 5
17, %
19, %
20, % 2, 20*5
21, % 3
24, % 2
27, % 
32, % 
36, % 2
49.
 \end{equation} The prime numbers are
$ 11, 17,19. $  
  For   hyperelliptic curves,
   Ogg \cite[Theorem 2]{Ogg} showed that  the following list of 19   numbers
 \begin{equation}\label{Ogglist}
 22,   23, 26, 28, 29, 30, 31, 33,  35, 37, 39, 40, 41, 46, 47, 48, 50, 59, 71
 \end{equation} 
  are the only possibilities of $N$ such that $ X_0(N)$   is 
  hyperelliptic. 
%For $N=37,  40,   48,$  the hyperelliptic involution is not  a product if Atkin--Lehner involutions. 
  And the following are all the pairs  $(N,p)$  such that $ X_0(N)$   is 
     $g\lb X_0(N)/\pair{w_p}\rb=0$, where $p|N$ is   a prime number:
   \begin{equation}\label{Ogglist1}
(22,11),(23,23),(28,7),(29,29),(31,31),(33,11),(41,41),(46,23),(47,47),(59,59), (71,71)
 \end{equation}

Recall  that by \Cref{lem:YKXK}, $X_0(N)=Y_0(N)(=Y_1(N))$, and   they and $X_1(N)$ are geometrically connected by \eqref{eq:pi0} .
    \begin{thm} \label{thm:auto-critical}
      Let $F=\BQ$, $B=\RM_{2,\BQ}$.

(1) The modular curve $ C=X_0(N)=Y_0(N)=Y_1(N)$   is auto-critical if the only if one of the following happens:

\begin{itemize}
\item 
   $N$ is  in   \eqref{g=0} (i.e., $g(C)=0$);
\item
  $N$ is  in   \eqref{g=1}   and $N\neq 49$,  in which case $g(C)=1$;
\item
    $N$ is  in   \eqref{Ogglist}   and $N\neq 37 $, in which case  $C$ is hyperelliptic.
      \item
    $N=54$, in which case  $C$ is non-hyperelliptic and of genus $4$. 
    \item
    $N=72$, in which case  $C$ is non-hyperelliptic and of genus $5$.

  \end{itemize} 
  
(2)    The modular curve $ C=Y(N)$   
  is auto-critical  if the only if one of the following happens:
  \begin{itemize}
\item 
  $N\leq 5$,  in which case $g(C)=0$;
\item
 $N=6$,  in which case $g(C)=2$.
 
  \end{itemize}

 (3)   The modular curve $ C=X_1(N)$     is auto-critical  if the only if one of the following happens:
\begin{itemize}
\item 
 $N\leq 10$ or $N=12$,  in which case $g(C)=0$;
\item
  $N=11 ,14,15 $, in which case
  $g(C)=1$;
 \item
  $N=13,16,18$, in which case $g(C)=2$.
   \end{itemize}

(4)    The modular curve $ C=X(N)$     is auto-critical  if the only if  
  $N\leq 6$, in which case $X(N)=Y(N)$.

  \end{thm} 
\begin{rmk} \label{exceptions}
 
  It might be interesting to point out that $X_0(37)$ in  \Cref{thm:auto-critical} (1) is the only case when $X_0(N)$ is hyperelliptic with  a hyperelliptic involution  not induced by an automorphism of the upper half plane $\BH$. This is a remarkable result of Ogg \cite[Theorem 1]{Ogg}. 
 
%(2)  The case $N=49$ of \Cref{thm:auto-critical} (1) reflects an observation of   Gross and Prasad   on test vectors  for trilinear forms \cite[Remark 7.6]{GP}.

 \end{rmk}

     The theorem   will be proved in this subsection.
    Before that,
   let us try to identify the   auto-critical Shimura curves that are not  subhyperelliptic.

  \begin{eg}\label{eg:auto-critical} 
   The base change  $X=X_0(54)_{\BC}$  is connected   and  non-hyperelliptic of genus $4$. 
It has a smooth affine model 
$$
y^3 = (x^3-1)(x^3 + 1) 
$$
 and  $\Aut(X)\cong\BZ/6\times S_3$, where $S_3$ is the symmetric group of three elements.
It has Group ID  $(36,12)$  in the Small Groups Library.  In particular, it is not a curve in our previous work \cite[Section 4]{QZ1}.
 Indeed, by \cite [Theorem 16]{Bar}, $\Aut(X)\cong \BZ/6\times S_3$.
      By  \cite[Table 4]{MSSV}, there is a unique such curve.  Its equation is found in \cite{Sw}. 
\end{eg} 
 \begin{eg}\label{eg:auto-critical0} 
 The base change  $X=X_0(72)_{\BC}$ 
is connected   and  non-hyperelliptic of genus $5$. 
 Moreover,
 $X$ is one of Wiman's curves discussed in our previous work  \cite[4.1.5]{QZ1}.
 Indeed,  by \cite [Theorem 16]{Bar},  $\Aut(X)\cong D_4\ltimes A_4\cong \GL_2(\BZ/4)$, 
 which 
has  Group ID  $(96,195)$ from the Small Groups Library
   where $D_4$ is the dihedral group of order $8$ and $A_4$ is the alternating  group of four elements.
 
\end{eg} 

To prove \Cref{thm:auto-critical}, we need some preparations.

 \begin{lem}\label{lem:auto-critical2}

   If  $  X_0(N) $ is auto-critical  and  there exists a prime $p|N$  such that  $g\lb X_0(p)\rb >0$ (that is, $p\neq 2,3,5,7,13$), then  $\Pi_p=\St$ for every $\Pi$ appearing in $X_0(N)$.
   Moreover,  if $N=p^rN'$ where $r$ is a positive integer and $p\nmid N'$, then    $g\lb X_0(pN')/\pair{w_p}\rb=0$. 
  
\end{lem} 
 \begin {proof} 
  Take   $\Pi'$ appearing in  $ X_0(p) $ (also  in  $ X_0(N)$) so that $\Pi'_{v\neq p}$ is a principal series.
  Since $g(X_0(1))=0$, $\Pi'_p$ is not a principal series. So $\Pi'_p=\St\otimes \pm 1$.
 By \Cref {lem:ep}  (2) at $p$,  $\Pi'_p=\St $ (otherwise $\Pi'^{\otimes 3}$ has nonzero  $\GL_2(\BA_\rf)$-invariant linear forms and
  $ X_0(N) $ is not auto-critical).  
   If  $\Pi_p\neq\St $, $\Pi_p$ is   a principal series, $\St\otimes -1$ or of conductor $>1$.
  By \Cref{lem:ep}  (1) (2) (4)  at $p$ respectively,
$\Pi\otimes \Pi'^{\otimes 2}$ has nonzero  $\GL_2(\BA_\rf)$-invariant linear forms. So 
   $X_0(N)$  is not auto-critical, a contradiction. 
  So $\Pi_p=\St $. In particular, this holds for representations appearing in $X_0(pN')$. By \Cref{minimal} (2),   $g\lb X_0(pN')/\pair{w_p}\rb=0$. 
  \end{proof}

   \begin{cor} \label{cor:auto-critical1}
Let  $C= X_0(N)$ be auto-critical. 
  For    a prime number   $ p|N$,  if $p\neq 2,3,5,7,13$, then $p\|N$, 
   and  $(N,p) $ is in \eqref{Ogglist1} or $N=p=11,17,19$.    In particular, if $N>71$, then $p= 2,3,5,7$ or $13$.
    \end{cor} 
    
      \begin{proof} Let  $N=p^rN'$ where $r$ is a positive integer and $p\nmid N'$. By Riemann--Hurwitz $$g(C)-1\geq p^{r-1}\lb g\lb X_0(pN')\rb-1\rb.$$
Now   let $p\neq 2,3,5,7,13$  (so that $g\lb X_0(p)\rb >0$ and thus $g\lb X_0(N)\rb\geq g\lb X_0(pN')\rb >0$), by \Cref{lem:auto-critical2} and \eqref{mn},
$g(C)=r g\lb X_0(pN')\rb .$

      Let us  prove $p\|N$, i.e., $r=1$.
           Assume   $r>1$. If $g\lb X_0(pN')\rb>1$, we then have $p^{r-1}\leq 2r-1$, which is 
 impossible. 
 If $g\lb X_0(pN')\rb=1$, then inspecting \eqref{g=1}, we find $N'=1$ and $p=11,17,19$. 
However, if  $N=p^r$ where $p=11,17,19$ and $r>1$, then $C$ is not auto-critical. 
      Indeed, using \textsf{LMFDB}, if $r>1$, we find $\Pi$ appearing in $C$ with $\Cond(\Pi_p)>1$. This is a contradiction to \Cref{lem:auto-critical2}.
    % (One may also apply \Cref{lem:auto-critical2} to other primes. However, these are enough for our propose.)
 Thus $r=1$. 
 The rest follows from \Cref{lem:auto-critical2}.
    \end{proof} 
    
  \begin{prop} \label{prop:>71}
Let $N>72$. Then  $C= X_0(N)$  is not auto-critical.
\end{prop} 
      
     \begin{proof} 
  Since $N>72$, by   \Cref{cor:auto-critical1}, if $ p|N$,   $p= 2,3,5,7$ or $13$. 
We want to use \Cref{lem:auto-critical}  and some explicit examples   (Example  \ref{eg1}-\ref{eg6}) to  remove the possibilities of $N$ step by step, and  
   show that no such $N$   could exist.
 
By  Example \ref{eg1} (and \Cref{lem:auto-critical}), $N|2^5\cdot 3^3\cdot 5^3\cdot 7^2\cdot 13 $.

By  Example \ref{eg2} (and that  $N>72$), $13\nmid N$, i.e., $N|2^5\cdot 3^3\cdot 5^3\cdot 7^2 $.

     By  Example \ref{eg3}, $7\nmid N$, i.e., $N|2^5\cdot 3^3\cdot 5^3  $.

          By  Example \ref{eg4},  $5^2\nmid N$, i.e., $N|2^5\cdot 3^3\cdot 5  $.

  By  Example \ref{eg5}, $5\nmid N$, i.e., $N|2^5\cdot 3^3   $. 
  
    By Example \ref{eg6}, no such $N$ could exist.
    \end{proof}

      \begin{prop} \label{prop:<71}
Let $N\leq 72$. Then  $C= X_0(N)$ is  auto-critical if the only if one of the following happens:

\begin{itemize}
\item 
   $N$ is  in   \eqref{g=0} (i.e., $g(C)=0$);
\item
  $N$ is  in   \eqref{g=1}   and $N\neq  27, 49$,  in which case $g(C)=1$;
\item
    $N$ is  in   \eqref{Ogglist}   and $N\neq 37$, in which case  $C$ is hyperelliptic;
  
\item $N=54,72$.

  \end{itemize} 
  
In particular,  $C$ is subhyperelliptic.
  \end{prop} 
  \begin{proof} 

Recall that by \Cref{cor:auto-critical1}, if $C$ is auto-critical, one of the following happens:
  $N$ is   in \eqref{Ogglist1}; 
or    $N=11,17,19$; or  prime factors of $N$  are in $2,3,5,7,13$.
 Also recall that in the examples in the proof of the  previous proposition, we have shown that
if $N$ is in  the following list, then $C$ is not auto-critical:
$$
64  \text{ in Example \ref{eg1}};  52, 65  \text{ in Eg. \ref{eg2}}; 42, 70,56, 63   \text{ in Eg. \ref{eg3}};   45, 60  \text{ in Eg. \ref{eg5}}. 
$$

From the last two facts just recalled, by enumerating, one knows   that  if $C$ is auto-critical, 
then
 \begin{itemize}
\item 
 $N$ is   in \eqref{Ogglist1}, or 
  \item    $N$ is  in   \eqref{g=1}, or 
    
 \item $N=26,35, 39, 40,    48,50, 54,72$.
   \end{itemize} 
%Note that \eqref{Ogglist1} and $26,30,35, 39, 40,    48,50$ exhaust    Ogg's list  \eqref{Ogglist} of  hyperelliptic curves, except $30,37$.   

                     Let us check them case by case.   
                  
                  First,  we use \textsf{LMFDB} to  check  that $\Pi_p=\St$ for every $\Pi$ appearing  in $C$ if
  $(N,p)$  is  in \eqref{Ogglist1} or 
the following list, where $N$ is in  \eqref{g=1}:  
 $$     
 (11,11),
(14,7),
(15,5),
(17,17),
(19,19),
(21,3) .
 $$ 
Thus $C$ is auto-critical by \Cref{lem:ep} (2).

Second, we need to check
 $N=  20, 24, 
27,   32, 
36,    
49,$ the rest of    \eqref{g=1}.  This is done  in Example \ref{eg:6egs}.
      
   Third, for $N=26,30,35, 39$ 
   which are square-free.  By \cite[Theorem 2]{Ogg}, the  hyperelliptic involution of $C$ is a product of Atkin--Lehner involutions. By \Cref{lem:hypauto-critical},  $C$ is auto-critical.

 Now there are only $N=  40,    48,50,54,72 $ left and we check their auto-criticalness in Example \ref{eg:5egs}.
    \end{proof}

   \begin{proof}[Proof of \Cref{thm:auto-critical}]
The computation of the genera  is standard using the genus formula (and 
 number of geometrically connected components), and is
omitted
 (or using  \eqref{eq:H10}, \eqref{mn}  and \cite[Section 4]{MY}).     Now let us prove the remaining of the theorem. 
   (1) was proved in \Cref{prop:>71}  and \Cref{prop:<71}, with the  non-hyperellipticity  
   of $X_0(54),X_0(72)$ by \eqref{Ogglist}.  
  (2) follows from  (1) and  \Cref{lem:square1}.

  For (3), by (1) and \Cref{lem:auto-critical},  we only need to check $N $ in the lists \eqref{g=0}, \eqref{g=1} and \eqref{Ogglist}, excluding $ 49,   37$, and $N=54, 72$,
   such that $g(C)>0$, that is $N=11$, or $N>12$. 
Using \textsf{LMFDB},  for $N\neq 11,13,14,15,16,18$, 
  we find      $\Pi_1,\Pi_2,\Pi_3 $  appearing in $C$ such that $\Pi_{i,p},i=1,2,$  (in fact Galois conjugate)  are principal series for any prime $p$ and $\Pi_1\otimes \Pi_2\otimes \Pi_3$ has trivial central character. 
   Then $C$ is not auto-critical by \Cref{lem:ep}  (1).
  For 
  $N =22,30$,  we find      $\Pi_1,\Pi_2 $  appearing in $C$ such that $\Pi_1$ has level  $N/2$ so that it is a principal series at the prime $2$, while $\Pi_2$ is a principal series at every prime $p\neq 2$. Then $C$ is not auto-critical by \Cref{lem:ep}  (1) applied to $\Pi_1^{\otimes 2}\otimes \Pi_2$  or $\Pi_1\otimes \Pi_2^{\otimes 2}$. 
   For $N= 13,16$, there are only  $\Pi_1,\Pi_2  $  (in fact Galois conjugate) appearing in $C$  and 
no   $\Pi_1^{\otimes n}\otimes \Pi_2^{\otimes 3-n} $ has trivial central character.        Then $C$ is  auto-critical by \Cref{lem:cc}.  
  For $N= 11,14,15$, there is only one $\Pi$ and $\Pi_p=\St$ at some $p|N$.  For $N=18$,  there is only one $\Pi$ and $\Pi_2=\St\otimes \chi$, for some  order $3$ characters   $\chi\neq 1$. Here we use \cite[Proposition 2.8 (2)]{LW} to identify the local component.  Then $C$ is  auto-critical by \Cref{lem:ep}  (2).

  (4) follows from  (2), \Cref{lem:auto-critical} and    the fact that  any $\Pi$ appearing in $X(N)$ has trivial central character (i.e., appears in $Y(N)$)   in this case, which can be found by using \textsf{LMFDB}. 
        \end{proof}

      \subsection{Shimura  curves over $\BQ$ (I)}

In this and next subsection, we     discuss  more general quaternionic Shimura  curves.
           Assume $F=\BQ$ and $B\neq \RM_{2,\BQ}$, i.e., $|S_B|>0$. 
  Still, let $$D=\prod_{p\in S_B}p,$$  which is the discriminant of $B$ and determines $B$.
               We want to classify auto-critical $X_*^A(N)$ and    $Y_*^A(N)$, $*\in \{0,1,\emptyset\}$, where $A$ is a product of all primes with positive powers in $S_B$.    If $A$ is  square-free, i.e., $A=D$, then
               $X_*^A(N)$ and    $Y_*^A(N)$  are simply denoted by $X_* (N)$ and    $Y_* (N)$ to simplify the notations. We discuss this case in this subsection.
Again, by \Cref{lem:YKXK}, $X_0(N)=Y_0(N)(=Y_1(N))$, and   they and $X_1(N)$ are geometrically connected by \eqref{eq:pi0} .

   Recall  that    (by \Cref{lem:degX1X}) $X_0(N)$ of genus $\leq 2$ is classified in  \cite[Table 4.1]{Voi}. 
   Also recall  that
      hyperelliptic $X_0(N)_{\BC}$ is classified in  \cite[Theorem 7,8]{Ogg1}.  
      In particular, there are only finitely many pairs $(D,N)$ such that $X_0(N)_{\BC}$ is  subhyperelliptic.
Moreover, we find that $X_0(N)_{\BC}$ is  subhyperelliptic with $N>1$ if and only if one of the following happens:
         $$D=6,\quad N=5,7, 11,13 ,  17,19,29,31,37;$$
$$D=10,\quad N=3,7, 11,13 ,   19,23;$$
$$D=14,22,\quad N=3,5;$$
$$D=15,\quad N=2,4;$$
$$(D,N)=(21,2),(26,3),(39,2).$$

    \begin{thm} \label{thm:Shauto-critical}
Let $F=\BQ$ and $B\neq \RM_{2,\BQ}$.

   (1)   The Shimura curve $  C=X_0(N)=Y_0(N)=Y_1(N)$
   is auto-critical if and only if one of the following happens:
         \begin{itemize}
 
  \item $   |S_B| =2$ and $C_{\BC}$ is subhyperelliptic;
    
    \item  $S_B=\{2,3,5,7\},\{2,3,5,11\}$
 and $N=1 $, in which case $C_{\BC}$ is of genus $5$ and is not subhyperelliptic.
\end{itemize} 
In particular,  there are only finitely many such curves as $B,N$ varies.

 (2)    The Shimura curve $ C=Y(N)$    with $N>1$ 
  is auto-critical  if the only if    $D=15$ and $N=2.$
In this case $C_{\BC}$ is   connected,  of genus $5$ and $g(C/\pair{w_{3}w_5})=0$.

  (3)  The Shimura curve $ C=X_1(N)$ with $N>1$     is auto-critical  if the only if one of the following happens:
          $$D=6,\quad N=5,7;$$
$$D=10,\quad N=3;$$
$$D=14,\quad N=3,5;$$
$$D=15,\quad N=2,4;$$
$$(D,N)=(21,2),(22,3),(26,3),(39,2),$$
   in which case,  $C_{\BC}$ is subhyperelliptic.

 (4)   The Shimura curve $ C=X(N)$ with $N>1$      is auto-critical  if the only if 
$D=15$ and $N=2,$
 in which case $X(N)=Y(N)$.

  \end{thm} 
  
  \Cref{thm:Shauto-critical} will be proved later in this subsection.  
  \begin{rmk}\label{readogg}
 (1) Recall  that by definition,  the natural maps 
 $X(1)\to X_1(1)\to X_0(1)$ and 
  $Y(1)\to Y_1(1)\to Y_0(1)$ are isomorphisms.  This is the reason for letting $N>1$ in \Cref{thm:Shauto-critical} (2)(3)(4)
  
  (2) The finite set of Shimura curves satisfying the first condition in (1) of the theorem can be read from   
\cite[Theorem 7,8]{Ogg1} and  \cite[Table 4.1]{Voi}. This is actually used in its proof.
\end{rmk}

   \begin{eg}\label{eg:Shauto-critical} 
 
Let $ S_B=\{2,3,5,7\},\{2,3,5,11\}$ and  $X=Y_0(1)_{\BC}$ which is connected non-hyperelliptic of genus $5$.
Then   $\Aut(X)= (\BZ/2)^4$.     In particular, it is not a curve in our previous work \cite[Section 4]{QZ1}, neither a curve appeared previously in this paper.
Indeed, as $\Aut(X)> (\BZ/2)^4$, if $\Aut(X)\neq (\BZ/2)^4$, then $|\Aut(X)|>4(5-1)$, i.e., $X$ has a large automorphism group in the sense of \cite{MSSV}. 
However, by \cite[Proposition 1]{Rot}, $\Aut(X)= (\BZ/2)^s$ for some $s\geq 4$. 
 Then by  \cite[Table 4]{MSSV},    checking the automorphism groups there divided by $16$, we find non of them is of this form with $s>4$. So  $\Aut(X)= (\BZ/2)^4$. From \textsf{LMFDB}, we find that the  family of  genus 5 Humbert curves are the only curves with this automorphism group. So $X$ is  a Humbert curve. Their criticalness is already proved in \cite{Lat}
\end{eg} 

Now we prepare to prove \Cref{thm:Shauto-critical}.
   The idea of the  proof  is pretty different from the proof of \Cref{thm:auto-critical}, as we are dealing with all quaternion algebras  simultaneously.
   We need first to narrow down the possibilities of the quaternion algebras.
   
   We need some lemmas that in fact hold for more general $F$. So let us not  assume  $F=\BQ$ for the moment. 
      Let    $K \subset  B^\times(\BA_\rf)$ be  an open compact subgroup such that      
                   \begin{equation*} 
                   K_{v}=\cO_{B_{v }}^\times \text{ for }v\in S_B.
 \end{equation*}
 Note that $K$ \textit{does not} contain $\BA_\rf^\times$.
 Let $ C=Y_{ K}:=X_K/\BA_\rf^\times.$  
 Recall that for  a representation $\Pi$   appearing in $C$,   $\Pi_v=\pm 1$ for $v\in S_B$.

We use two technical notions.  We say that   $T\in \{\pm 1\}^{{S_B}}$ appears in $C$ if $T=\lb\Pi_{v}:v\in {S_B}\rb$ for some $\Pi$   appearing in $C $.
Say $T$ is odd (resp. even) if  the product of its components is  $-1$ (resp. $1$).
 By definition, we have the following lemma.
      \begin{lem}\label{lem:Shauto-critical} 
(1) Assume that   $ C $ is auto-critical and $K$ is maximal. Then for  $T_1,T_2,T_3\in \{\pm 1\}^{{S_B}}$   appearing in $C$, 
the component-wise product  $T_1\cdot T_2\cdot T_3\neq(1,...,1) $, 
 
 (2)  Assume  that there exists a  subset $S\subset S_B$ such that 
  $\prod_{v\in S}\Pi_v=-1$ for all $\Pi$ appearing in $C$, then $g\lb C/\pair{w}\rb=0$, where
  $w=\prod_{v\in S} w_v.$
\end{lem}

  We have the following  three lemmas which are partial analogs of
\Cref{lem:AL} and \Cref{lem:auto-critical2}.
   \begin{lem}\label{lem:Shauto-critical1} 
   Assume that   $ C $ is auto-critical. 
   Assume   that $   |{S_B}| \leq 3$,   $K$ is maximal and
  $ C $ is auto-critical.
  Then $g\lb C/\pair{w}\rb=0$, where  $w=\prod_{v\in S} w_v$ for some 
subset $S\subset {S_B}$.    \end{lem} 
     \begin{proof}    If $ |{S_B}|=0$,  this is a special case  \Cref{lem:AL} with $\fN, \fa  =\cO_F$.      
 
      If $ |{S_B}|=1$ and ${S_B}=\{v_0\}$, 
      by \Cref{lem:Shauto-critical} (1), 
           $\Pi_{v_0}=-1$ for $\Pi$   appearing in $C$. 
    Then  the lemma follows from   \Cref{lem:Shauto-critical} (2).
 
Assume $ |{S_B}|=2$ and ${S_B}=\{v_1,v_2\}$. 
Since the component-wise product  $(1,1)\cdot(1,1)\cdot(1,1)=(1,1) $, 
 by \Cref{lem:Shauto-critical} (1), 
  $  (1,1)$ 
  does not  appear  in $C$. 
  And since the component-wise product $(1,-1)\cdot(-1,1)\cdot(-1,-1)=(1,1) $,  by \Cref{lem:Shauto-critical} (1), 
one of  $(1,-1), (-1,1), (-1,-1)$ does not   appear  in $C$.
Then  the lemma follows from   \Cref{lem:Shauto-critical} (2).

 If $ |{S_B}|=3$ and ${S_B}=\{v_1,v_2,v_3\}$,    $(\pm 1)^{{S_B}}$  has $8$ elements.
 Let us first   consider the four  even ones.
By \Cref{lem:Shauto-critical} (1),   $ (1,1,1)$ 
 does not  appear  in $C$.   Since the component-wise product of $(1,-1,-1), (-1,1,-1)$ and $ (-1,-1,1)$ is $(1,1,1)$, one of them    does not  appear  in $C$. Then up to permutation of $\{v_1,v_2,v_3\}$, we have three cases: first,    non of them  appears  in $C$;  second,  only $(1,-1,-1)$    appears  in $C$; finally both 
 $(1,-1,-1), (-1,1,-1)$    appear  in $C$ . 
 In each case, let us consider the four odd triples left:
 $$a=(-1,-1,-1),b=(1,1,-1),c=(1,-1,1),d=(-1,1,1).$$ 
 $(1,-1,-1)$ and $(-1,1,-1)$ 
 
  In the first case, only (some of) the odd triples appear in $C$, and   the lemma follows from  \Cref{lem:Shauto-critical} (2)
  and the oddness.
 In the second case, by considering   component-wise products with $(1,-1,-1)$ and using by \Cref{lem:Shauto-critical} (1), we see that $a,d$ as well as  $b,c$   do not   appear in $C$ together. For each of the rest possibilities, 
  it is easy to prove  the lemma by  \Cref{lem:Shauto-critical} (2).
     In the final case, by considering   component-wise products with $(1,-1,-1)$ (resp. $(-1,1,-1)$)    we see that $a,d$ as well as   $b,c$ (resp. $a,c$ as well as  $b,d$) do not   appear in $C$ together. For each of the rest possibilities, 
   it is easy to prove  the lemma by  \Cref{lem:Shauto-critical} (2).
         \end{proof} 
      
             \begin{lem}\label{lem:Shauto-critical2}   
 
  Assume  that  $   |S_B| \leq2$ and    $ C $ is auto-critical. 
  Further assume that $g\lb Y_{K'}\rb>0$ for  (one and thus all) $K'\subset  B^\times(\BA_\rf)$ maximal.
   Then $g\lb C/\pair{w}\rb=0$,  where  $w=\prod_{v\in S} w_v$ for some 
subset $S\subset S_B$.  
\end{lem} 
 \begin {proof} 
We may assume $K\subset K'$. Let us deal with the case  $   |S_B| =2$, say $S_B=\{v_1,v_2\}$. 
The other cases are easier and omitted.  

Let   $\Pi'$ be appearing in  $ Y_{K'}$ so that $\Pi'_{v\not\in S_B}$ is a principal series. 
Let $\Pi $ be appearing in $C$.
Then since the component-wise product  $\lb\Pi'_{v_1},\Pi'_{v_2}\rb^{\cdot2} \cdot(1,1)=(1,1) $, 
by \Cref{lem:ep}  (1),
$\Pi\otimes \Pi'^{\otimes 2}$ has nonzero  $B^\times(\BA_\rf)$-invariant linear forms
if 
  $\lb\Pi_{v_1},\Pi_{v_2}\rb= (1,1)$. So $(1,1)$
  does not  appear  in $C$. 
  The rest of the proof is similar to the one  for the case $   |S_B| =2$ in \Cref{lem:Shauto-critical1} and omitted.
 \end{proof}

         \begin{lem}\label{lem:Shauto-critical3}   
  Assume  that   $   |S_B|=2$, $   Y_0(\fp) $ is auto-critical for a prime $\fp$ and $g\lb Y_0(\cO_F)\rb=0$. 
       Then $g\lb Y_0(\fp)/\pair{w}\rb=0$, where  $w=\prod_{v\in S} w_v$ for some 
subset $S\subset S_B\cup \{\fp\}$.
\end{lem} 
 \begin {proof} 
Since $g\lb Y_0(\cO_F)\rb=0$, for $\Pi$ appearing in $Y_0(\fp)$   $\Pi_\fp=\St\otimes \pm 1$, and determined by the Atkin--Lehner sign (which is $\mp1$).  
     The rest of  the proof is similar to the one  for   \Cref{lem:Shauto-critical1}.
  \end{proof} 
  
  Now let $F=\BQ$. Let $B\neq \RM_{2,\BQ}$, i.e., $|S_B|>0$. 
   Recall $S_ N:=\{p:p|N\}$.
   
   \begin{lem}\label{lem:Shauto-critical5}   
  Let  $N$ be square-free and prime to $D$. If  there are (not   necessarily  different)  newforms  $f_i,i=1,2,3,$   of level $DN$  with trivial central character whose   Atkin--Lehner signs $s_{i,p},p\in S_B\cup S_N ,$
 satisfy   
  $\prod_{i=1}^3s_{i,p}=-1$ for $p\in S_B $ and   $\prod_{i=1}^3s_{i,p}=1$  for $p\in S_N$, then $Y_0(N)$ is not auto-critical.

\end{lem} 
 \begin {proof}  
By    \Cref{minimal} (2),  for $\Pi_i,i=1,2,3,$ appearing in 
$Y_0(N)$, we can
 apply \Cref{lem:ep} (2) to conclude that   $  \Hom_{B_p^\times}\lb \Pi_{1,p}\otimes \Pi_{2,p}\otimes \Pi_{3,p},\BC\rb\neq 0$ for $p\in S_N$.
Apply  \Cref{JLpm1} to conclude this for $p\in S_B$.
     \end{proof}

  The proof of   \Cref{thm:Shauto-critical}
 (1) will start with the following. 
     \begin{prop}\label{cor:Shauto-critical2}    
         Assume  that  $   |S_B| =2$ and $ C =Y_0(N)$ is auto-critical. Then    $g\lb C/\pair{w}\rb=0$, where    $w=\prod_{v\in S} w_v$, for some
subset   $S\subset S_B\cup S_N $.  In particular, $C$ is subhyperelliptic 
\end{prop} 
  The proposition will be proved after a lemma.
  
By the classification  $Y_0(N)$ of genus $\leq 2$   in  \cite[Table 4.1]{Voi},  we have   \begin{equation}
\label{g(Y_0(N))>0}
g(Y_0(N))>0 
\text{ if both }D,N>1;
\end{equation}  
 \begin{equation}
\label{g(Y_0(1))=0}
g(Y_0(1))=0
\text{ if and only if } D=6,10,22.
\end{equation}

    \begin{lem}\label{lem:Shauto-critical4}   
  Assume    $g\lb Y_0(1)\rb=0$ and  $  Y_0(p^n) $ is auto-critical for a prime $p$. Then  $n=1$ and
   $g\lb Y_0(p^n)/\pair{w}\rb=0$, where  $w=\prod_{v\in S} w_v$ for some 
subset $S\subset S_B\cup \{p\}$. 
%In particular,  $(D,p)$ is as   above \Cref{thm:Shauto-critical} (with $p$ being one of the $N$ following the corresponding $D $ in $6,10,22$).  
\end{lem} 
 \begin {proof} 

  By  \Cref{lem:auto-critical} and  \Cref{lem:Shauto-critical3}, 
$g\lb Y_0(p)/\pair{w}\rb=0$, where  $w=\prod_{v\in S} w_v$ for some 
subset $S\subset S_B\cup \{p\}$. 
  We only needed to show that $n=1$.
  Before this, we note that by \eqref{g(Y_0(1))=0}, $D=6,10,22$. 
And as $g\lb Y_0(p)/\pair{w}\rb=0$, $p$ is as  above \Cref{thm:Shauto-critical} (i.e.,  one of the $N$'s) following a given $D$.

  If $n> 1$, we use \textsf{LMFDB} to find  $\Pi$ appearing in $Y_0(p^2) $ (also in $Y_0(p^n ) $) such that  $\Pi_v=1$  for $v\in S_B$  (note the change of sign under   Jacquet--Langlands correspondence, see \Cref{JLpm1})  and $\Cond(\Pi_p)=2$.   Let $\Pi'$ appear in $Y_0(p)$ (as $g(Y_0(p))>0=g(Y_0(1))$ by \eqref{g(Y_0(N))>0}). Then $\Pi'^{\otimes 2}_v\otimes \Pi_v=1$  for $v\in S_B$, as $\Pi_v=1$.
Since  $\Cond(\Pi_p)= 2$, by \Cref{lem:ep} (4), $Y_0(p^n) $ is not auto-critical, a contradiction.
    \end{proof}  

       \begin{proof}[Proof of  \Cref{cor:Shauto-critical2}]  Assume the contrary, i.e., $g\lb C/\pair{w}\rb\neq0$, where    $w=\prod_{v\in S} w_v$, for any
subset   $S\subset S_B\cup S_N$.  
By  \Cref{lem:Shauto-critical2}, 
$g\lb Y_0(1)\rb=0$. 
  We will prove that $|S_N|=1$. Then we have a contradiction  to \Cref{lem:Shauto-critical4}.

First,  we claim  that if $|S_N|>1$, then for $p\in S_N$, one of the following happens: 
$$D=6, \quad p=5,7, 13;
$$
$$D=10, \quad p=3,7. 
$$ 

Second,
assuming the claim, we find that $C$ can not be auto-critical.  Indeed,  for each $D$ in the second step, and $p_1\neq p_2$    two corresponding prime numbers there,  
using \textsf{LMFDB} \cite{LMFDB}, we find (not   necessarily  different)  $f_i,i=1,2,3,$  newforms   of level $Dp_1p_2$  
as in \Cref{lem:Shauto-critical5} (with $N=p_1p_2$), so that $Y_0(p_1p_2)$ is not auto-critical.  (For example, if $D=6,p_1p_2=7\cdot13,$ take newforms labeled         as $546.2.a.b$, $546.2.a.c$  and $546.2.a.j$.) 
  By  \Cref{lem:auto-critical}, $C$ is not auto-critical.

Third, we prove the claim.   As $g\lb Y_0(1)\rb=0$, $D$ is as in  \eqref{g(Y_0(1))=0}.
Assume
$(D,p)$  is not  as in the claim. Then by \Cref{prop:ALg=0}, 
there exists $\Pi$ appearing in $Y_0(p)$ such that  $\Pi_v=1$  for $v\in S_B$. 
  Let $p\neq q\in S_N$ and let $\Pi'$ appear in $Y_0(q)$ (as $g(Y_0(q))>0$ by \eqref{g(Y_0(N))>0}).  Then $\Pi'^{\otimes 2}_v\otimes \Pi_v=1$  for $v\in S_B$, as $\Pi_v=1$. Since  $ \Pi_q,\Pi'_p $ are principal series, by \Cref{lem:ep} (1), $\Pi'^{\otimes 2}_v\otimes \Pi_v$
  has   nonzero $B_v^\times$-invariant linear forms. So
  $C$ is not auto-critical, a contradiction.
   \end{proof}

        \begin{proof} [Proof of \Cref{thm:Shauto-critical}]
 (1) Assume $|S_B|=2$. By \Cref{cor:Shauto-critical2}, we only need to prove the ``if" part. If $C_{\BC}$ is hyperelliptic, by inspecting 
\cite[Theorem 7,8]{Ogg1} which gives the hyperelliptic involutions,   we find that
\Cref{lem:hypauto-critical} applies so that $C$ is auto-critical.
   If $C_{\BC}$ is  elliptic as listed in   \cite[Table 4.1]{Voi}, then there is only one representation $\Pi$ appearing in $C=Y_0(N)$. 
   By inspecting \cite[Table 4.1]{Voi}, $N$ is square-free (in fact  a prime).
   By inspecting  \cite[Proposition 4.1]{Rie}, we find that  the total Atkin--Lehner quotient of $C$ over $S_B\cup S_N$ has genus $0$. So
   $\Pi_v=-1$ at some $v\in S_B$ or  has Atkin--Lehner sign $-1$ at  $v\in S_N$ (one may also check this fact directly using \textsf{LMFDB}).
So $C$ is auto-critical by \Cref{lem:ep} (2).
   If $C_{\BC}$ is rational, it is trivial.    
   
  Assume $|S_B|>2$.    Let $a=\prod_{p\in S_N}p.$  
       By  \Cref{lem:auto-critical}, $Y_0(a)$ is auto-critical, and then $ Q^B_0(1)_{a} $ is auto-critical.
        Then      by  \Cref{lem:AL},  $ Q^B_0(1)_{a} $ has genus $0$.
        As \cite[Proposition 4.1]{Rie} classified all such $B$ and $a$, we can using   \textsf{LMFDB}
   and \Cref{lem:Shauto-critical5}
       to check that $Y_0(a)$ is not auto-critical case by case, except when $S_B=(2,3,5,7),(2,3,5,11)$ and $a=1$. Thus $N=1$. 
When $S_B=(2,3,5,7),(2,3,5,11)$, 
 use        \Cref{JLpm1}         to check that  $Y_0(1)$ is in fact  auto-critical.  
      Moreover,  by   \cite[Theorem 7]{Ogg1},   $C$ is not subhyperelliptic.
  The genus is exactly the number of representations appearing in $C$, which is $5$ by     \cite{LMFDB}. 
 
 (2)  The first part follows from  (1) and  \Cref{lem:square1}.     The hyperelliptic involution (and genus) can be read from \textsf{LMFDB}, with 
 \eqref{eq:H10}
 (and   dimension formula for invariants by a compact subgroup   of $\GL_2(\BQ_p)$  
 \cite[Proposition 4.3]{MY}).
  It is geometrically connected by \eqref{eq:pi0}  or by the genus.

   (3) For the auto-criticalness, by (1) and \Cref{lem:auto-critical},  we only need to check $D,N $ as in (1), that is,
$D,N$ as above the theorem. 
Note that if every newform    of level $c$  such that  $D|c,c|DN$ and with central character unramified over $S_B$ in fact has trivial central character (for example, if $N=2$), then  
$X_1(N)\to X_0(N)$  is an isomorphism by the same reasoning as in the remark after \Cref{lem:YKXK}.
 Then $X_1(N)$ is auto-critical.
  The rest is verified using \textsf{LMFDB} and \Cref{lem:X1N}. 
  If $(D,N)$ is   in (3), \Cref{lem:X1N} (1)  is applicable. In fact, we only need to apply  it in the following three cases where $X_1(N)\to X_0(N)$  is not an isomorphism.
If $(D,N)=(6,5)$,   take  $n_2= 1,n_3=3$. 
If $(D,N)=(6,5)$,  take  $n_2=n_3=1$. 
If $(D,N)=(14,5)$, use \Cref{lem:X1N} (1) and take  $n_2=n_7=1$. 
 If $(D,N)$ is not in (3), \Cref{lem:X1N} (2)  is applicable.
In fact, we only need the ``particular" part  of \Cref{lem:X1N} (2) which is very handy, expect the following two cases.
 If $(D,N)=(6,13)$, use \Cref{lem:X1N} (2) and take  
  $f_1,f_2,f_3$ to be   newforms labeled         as $78.2.e.b,78.2.i.a,78.2.i.b$. 
  If $(D,N)=(10,7)$, use \Cref{lem:X1N} (2) and take  
  $f_1,f_2,f_3$ to be   newforms labeled         as $70.2.e.b,70.2.e.c,70.2.e.d$. Here, each label correspond a Galois orbit, and we take the newform listed on \textsf{LMFDB}, but not any other Galois conjugate.

 (4) By    (2) and \Cref{lem:auto-critical}, we only need to check the case $D=15,N=2$. 
  We have the isomorphism by the same reasoning as in \Cref{lem:YKXK}.
        \end{proof}

 The following lemma helps us to use \textsf{LMFDB}.
     \begin{lem}\label{lem:X1N}   
     Let the  newforms in this lemma have central characters unramified over $S_B$.

(1)  If there exists $n_p,p\in S_B$,
such that
 for any   newform $f$     of level $c$  with  $D|c$ and $c|DN$, 
   its  $p$-the Fourier coefficient    $   a_p(f)$  satisfies 
 $\prod_{p\in S_B}a_p(f)^{n_p}=-1$, then $X_1(N)$ is auto-critical and $X_1(N)_{\BC}$ is subhyperelliptic.

 (2)  Assume there exists a newform $f_1$ with  central character  of conductor $\ord _qN$ for   $q\in S_N$,
  If there exist
two   newforms  $f_2,f_3$     of level   $DN$,  
such that
 for $p\in S_B$, their $p$-the Fourier coefficients    $   a_p(f_i)$'s  satisfy 
 $\prod_{i=1}^3 a_p(f_i) =1$,
  then $X_1(N)$ is not auto-critical. 
In particular, if there exists
a   newform  $f_3$     of level   $DN$, such that
 for $p\in S_B$,  $a_p(f_3) =1$,  then $X_1(N)$ is not auto-critical.

      \end{lem} 
  \begin{proof}
(1)   
At $p\in S_B$, 
 we have $B_p^\times/\cO_{B_{p }}^\times \cong \BZ$ acts on $X_1(N)$
as in \eqref{DZ/2'}. Let  $w_p$ be  a generator.
By \eqref {chiap} and \eqref{JLchi},  $\prod_{p\in S_B}a_p(f)^{n_p}=-1$ implies that   $\prod_{p\in S_B}w_p^{n_p}$ acts as $-1$ on
$H^{1,0}(X_1(N)_{\BC})$
 (see \eqref{eq:H10}). So  $X_1(N)_{\BC}$ is subhyperelliptic.   
To show that $X_1(N)$ is auto-critical, proceed as in the second paragraph of the proof of \Cref {lem:hypauto-critical}.

(2) By \Cref{minimal} (4),
$\Pi$ appearing in  $X_1(N)$ corresponding  to $f_1$
 is a principal series at $q\in S_N$. The rest of the proof is similar to  (1) except that we also use \Cref{lem:ep} (1).  
 The ``particular" part follows  as we can take $f_2$ to be the complex conjugate of $f_1$.
\end{proof}

                \subsection{Shimura  curves over $\BQ$ (II)}\label{ShQII}  
              Still   let $F=\BQ$ and $B\neq \RM_{2,\BQ}$.
           Now we  consider Shimura  curves  with more general  levels at $S_B$, i.e., $X_*^A(N)$ and $Y_*^A(N)$.
          
\begin{thm} \label{thm:ShAauto-critical}
Let $F=\BQ$ and $B\neq \RM_{2,\BQ}$.
Let $A$ be a product of all primes with positive powers in $S_B$.
Assume that $A$ is not square-free.
 
 (1)    The Shimura curve $C=Y_1^A(N)=Y_0^{A}(N)$ is auto-critical 
  if the only if 
  one of the following happens:
         \begin{itemize}
 
  \item[(1.a)] $A$ divides one of the following $$  3^3\cdot 2, 2^4\cdot 3^2,    2^2\cdot 5^2, 2^4\cdot 11,2^2\cdot 7,2^2\cdot 17,2^2\cdot 29,2^2\cdot 41,5^2\cdot 3,3^2\cdot 5,3^2\cdot 7, 3^2\cdot 19,3^2\cdot 31, 7^2\cdot 2$$  
  and $N=1$;

\item [(1.b)] 
 $A=2^2\cdot 3, N=5,7,13$,  
 \item [ ]  $A=2^2\cdot 5,N=3$,
\item [ ]  $A=3^2\cdot 2,N=5,13$, or 
\item[ ]  $ A= 3^2\cdot 7,N=2.$

\end{itemize}

 (2)    The Shimura curve $ C=Y^A(N)$    with $N>1$ 
  is not auto-critical.
 
  (3)  The Shimura curve $ C=X_0^A(N)$      is auto-critical  if the only if      $A,N$ is as in (1) and $A\neq 2^4\cdot 3^2. $

  (4)  The Shimura curve $ C=X_1^A(N)$ with $N>1$      is auto-critical  if the only if    
$ A= 3^2\cdot 7,N=2$ or   $ A= 2^2\cdot 5,N=3$,
in which case $X_1^A(N)=X_0^A(N)$.

  (5) The Shimura curve $ C=X^A(N)$ with $N>1$     is not auto-critical.

  \end{thm} 
   
\begin{rmk} 
   (1) We let  $N>1$ in \Cref{thm:ShAauto-critical} (2)(4)(5) by the same reason as in the remark after \Cref{thm:Shauto-critical}.

(2)
In \cite{Qiufinite} of the computation of the examples in \Cref{newforms in the database LMFDB}, as a byproduct,  we find that most of the auto-critical Shimura curves in the theorem are subhyperelliptic (as in \Cref{thm:auto-critical} and  \Cref{thm:Shauto-critical}). 
For the  other auto-critical Shimura curves, we in fact can identify them as   the examples in \Cref{Standard odular curves}. We find that the non-subhyperelliptic ones have all appeared in our previous work  \cite{QZ1} so that we do not discuss them here or in \cite{Qiufinite}.
 And this is not a surprise  since the levels at primes  $v\in S_B$  are normal subgroups of $B_v^\times$.   
 
\end{rmk}
 \Cref{thm:ShAauto-critical} will be proved in this subsection after  some lemmas.
   
   Similar to the   use of \Cref{prop:ALg=0} about genus $0$ Atkin--Lehner quotients  in the last subsection,  we will need the following lemma.
   \begin{lem}\label{lem:ALg=0}
Assume  that $p\neq q$ are primes. 

(1)   The genus $g\lb Y_0^{pq}(1)/\pair{w_q}\rb=0$ if and only if 
\begin{align*}(p,q)=& (2,3),(3,2),(2,5),(5,2),(2,11),(11,2),(11,3),(2,17),\\
&(2,29),(2,41),(5,3),(3,7),(3,19),(3,31),  (7,2),(23,2).
\end{align*}
(The ordering of the pairs here is for some convenience in  later  use.)

 %$$(2,3),(3,2),(2,5),(5,2), (7,2),(5,3),(3,7),(2,11),(11,2),(11,3),(2,17),(23,2),(3,19),(2,29),(2,41),(3,31).$$

 (2)  The genus    $g\lb Y_0^{pq}(N)/\pair{w_q,w_N}\rb=0$ for a prime $N$ if and only if  one of the following holds:
 $$(p,q)=(2,3),\quad N=7,13,19,43;$$   
     $$(p,q)=(3,2),\quad N=5,13,17;$$
 $$(p,q)=(2,5),\quad N=3,7,13;$$  
   $$(p,q)=(5,2),\quad N=3,19;$$    $$(p,q,N)=(2,11,3),   (2,17,3)
,(5,3,7)   ,(3,7,2).$$

  \end{lem}
     \begin{proof}  
           (1) If $g\lb Y_0^{pq}(1)\rb>1$, we use   \cite[Theorem 7,8]{Ogg1}, which  gives a classification of hyperelliptic involutions.   If $g\lb Y_0^{pq}(1)\rb=0$, we can find $(p,q)$ in \cite[Table 4.1]{Voi}.
                       If $g\lb Y_0^{pq}(1)\rb=1$, we first find  $(p,q)$ in \cite[Table 4.1]{Voi}, and then
 use \textsf{LMFDB} to find the corresponding newform $f$ of level $pq$. By \Cref{JLpm1}, 
 $g\lb Y_0^{pq}(1)/\pair{w_q}\rb=0$ if and only if the Atkin--Lehner sign of $f$ at $q$ is $1$.
 The latter information is read from \textsf{LMFDB}.

  (2) 
   Since $g\lb Y_0^{pq}(N)/\pair{w_q,w_N}\rb=0$ implies that $ g\lb Y_0^{pq}(N)/\pair{w_p,w_q,w_N}\rb=0$,
we only need to check $(p,q,N)$'s listed in \cite[Proposition 4.1]{Rie}.
   Similar to  (1)$\Rightarrow$(4) of \Cref{lem:AL3},  we have  $g\lb Y_0^{pq}(N)/\pair{w_q,w_N}\rb=0\Rightarrow g\lb Y_0^{pq}(1)/\pair{w_q}\rb=0$. 
  Thus $(p,q)$ must be as in (1).  
  Now we
 use \textsf{LMFDB} to check the remaining possible $(p,q,N)$'s by using a criterion: $g\lb Y_0^{pq}(N)/\pair{w_q,w_N}\rb=0$ 
 if and only if for any  newform $f$ of level $pqN$ with  
 the Atkin--Lehner sign $1$ at $N$, its Atkin--Lehner sign   at $q$ is  also $1$.
This criterion   follows from \Cref{minimal} 
 \Cref{JLpm1} (note the change of sign).
     \end{proof}

     By \Cref{lem:auto-critical}, if $Y_0^A(N)$ is auto-critical, then $Y_0(N)$  is auto-critical. 
         And  we have classified auto-critical $Y_0(N)$ in the last subsection (\Cref{thm:Shauto-critical} and the discussion above it). In particular, $|S_B|=2,4$. %Moreover,    if  $|S_B|=2$, a auto-critical $Y_0(N)$ is subhyperelliptic 
     The proof of   \Cref{thm:ShAauto-critical}
  starts with the following lemma, which narrows down the possible exponents in $A$.
 
   \begin{lem}\label{lem:ram1}
Let  $p\neq q$ be primes. 

(1)   If $(p,q)\neq (2,7),(3,5)$ and $Y_0^{p^2q}(1)$ is  auto-critical, then  $g\lb Y_0^{pq}(1)/\pair{w_q}\rb=0$.

 (2) If  $(p,q,N)\neq(2,3,5)$ and
$Y_0^{p^2q}(N)$ is  auto-critical for a prime $N$, then  $g\lb Y_0^{pq}(N)/\pair{w_q,w_N}\rb=0$.

   (3)  If  
$Y_0^{p^2q}(N)$ is  auto-critical where $N>1$ (but  a priori not necessarily prime), then $(p,q,N)$ is $ (2,3,5)$,   or
as in \Cref{lem:ALg=0} (2). In particular, $N$ is prime in this case.

  \end{lem}
     \begin{proof}  
     Let $\cR$ be the set of   representations  appearing in $  Y_0^{p^2q}(N) $ is that is minimal  of  conductor $2$ at $p$.

 (1) If $\{p,q\}=\{2,3\}$, $g\lb Y_0^{pq}(1)\rb=0$ (by \textsf{LMFDB} or \cite[Table 4.1]{Voi}) and there is nothing to prove. 
Otherwise, under the assumption in (1),  we claim that 
 $\cR\neq \emptyset$. Indeed,
by the discussion above the lemma,   we can check the claim  case by case using \textsf{LMFDB}.
Let $\Pi_{2}\in \cR$.
   Assume $g\lb Y_0^{pq}(1)/\pair{w_q}\rb\neq 0$. Then  there is 
$\Pi_{1}$   appearing in $Y_0^{p^2q}(N)$ minimal
 of levels $pq$  such that $\Pi_{1,q}=1$.
 As $ \Pi_{2,q}^{\otimes 2}= 1$, $\Pi_{1} \otimes\Pi_{2}^{\otimes 2}$
 has nonzero $B_q^\times$-invariant linear forms. 
 As $\Pi_{1,p}=\pm1$ and $\Pi_{2,p} $ is minimal  of  conductor $2$, by  \Cref{lem:Dn} (1) to $\Pi_{1} \otimes\Pi_{2}^{\otimes 2}$, $C$ is not auto-critical, a contradiction.

     (2)   is similar to (1), except now we also use      \Cref{lem:ep} (1)(2) at $N$.

     (3)  By the discussion above the lemma,  $N$ is a prime unless $(p,q,N)=(3,5,4),(5,3,4)$.  
  If    $(p,q,N)=(3,5,4),(5,3,4)$,  then by      \Cref{lem:auto-critical},  $Y_0^{p^2q}(2)$ is  auto-critical. This is a contradiction to \Cref{lem:ALg=0} (2).
  Thus $N$ is a prime. (3) then follows from (2).
     \end{proof}
        
         \begin{lem}\label{lem:ram2}     Assume  that $p,q $ are distinct primes  and some representation $\Pi$ appearing in $  Y_0^{p^2q^2} (1) $ is minimal of conductor $2$ at both $p,q $.
  If   $Y_0^{p^2q^2} (N) $ is  auto-critical, then $g\lb Y_0^{pq}(N)\rb=0$. In this case, $N=1$ and
   $\{p,q\}=\{2,3\},\{2,5\},\{2,11\}$.

 \end{lem}
     \begin{proof}   If $g\lb Y_0^{pq}(N)\rb\neq 0$,
     some
     $\Pi'$ appears in $Y_0^{pq} (N)$ so that $\Pi'_{p}=\pm1, \Pi'_q=\pm1$. 
     Then  we can apply     \Cref{lem:Dn}  (1)  (combined with   \Cref{lem:quat1}) 
      $\Pi^2\otimes\Pi'$ has nonzero  $B_p^\times\times B_q^\times$-invariant   linear forms.
          As $\Pi_v$ is a principal series at $v|N$, by  \Cref{lem:ep} (1),  $\Pi^2\otimes\Pi'$ has nonzero  $B^\times(\BA)$-invariant   linear forms,      a contradiction.  
  The second part of follows from 
 \eqref {g(Y_0(N))>0} and \eqref {g(Y_0(1))=0}.            \end{proof}
            Using \Cref{lem:ram2}  and \textsf{LMFDB}, we have the following.

       \begin{cor}\label{eg:23513} 
      (1)   For $A=2^2\cdot 7^2, 3^2\cdot 5^2$, $Y_0^{A} (1) $ is not auto-critical

(2)   For $A=2^2\cdot 3^2,   N=5,13$, $Y_0^{A} (N) $ is not auto-critical.

 \end{cor}

  \begin{proof}[Proof of \Cref{thm:ShAauto-critical}] %The strategy is a similar to the one in the proof of \Cref{thm:auto-critical}.
  %We want to use \Cref{lem:auto-critical}  and some explicit examples   to  pin down all possible $A,N$.
    
  (1) 
  Assume $Y_0^A(N)$ is auto-critical.
   By the discussion above \Cref{lem:ram1}, we only need  to consider $|S_B|=2,4$. 
  
  First, we assume $S_B=\{p,q\}, N=1$, that is (1.a). Then $A=p^nq^m$ with $n\geq m,n>1$. By 
\Cref{lem:auto-critical} and \Cref{lem:ram1} (1),    $(p,q)$ can only be  a pair as in  \Cref{lem:ALg=0} (1)  or $(2,7),(3,5)$.  
Moreover, we claim: except for $S_B=\{p,q\}=\{2,3\},\{2,5\},\{2,11\}$,  $A$ can only be of the form $p^nq$, $n>1$.
 Indeed, if $A=p^nq^m$ with $n\geq m>1$, then   $ (q,p)$  can only be  a pair as in  \Cref{lem:ALg=0} (1)  or $(2,7),(3,5)$.  So $S_B=\{p,q\}=\{2,3\},\{2,5\},\{2,11\}$ or $ \{2,7\},\{3,5\}$. 
However, \Cref{eg:23513} (1) excludes the possibility of $ \{2,7\},\{3,5\}$.  

For each possible pair $(p,q)$ in the last paragraph,
to pin down auto-critical $Y_0^A(1)$, we proceed as in the proof of  \Cref{prop:>71}.
We   check some examples $A$'s    in Example \ref{eg:ShA1}-\ref{eg:ShA5}. And we follow  the ordering of possible $(p,q)$'s in \Cref{lem:ALg=0} (1) with adding  $(2,7)$ between $(11,3)$ and $(2,17),$ and adding
 $(3,5)$ between $(5,3)$ and $(3,7)$.
 For example, $S_B=\{2,3\},\{2,5\},\{2,11\}$ are the first three examples, each include two  pairs $(p,q),(q,p)$.
Then (1.a) follows from \Cref{lem:auto-critical}.

Second, we assume $S_B=\{p,q\}, N>1$, that is (1.b).
By \Cref{lem:auto-critical}, $A$ must be as in (1.a). 
 Moreover, if $2^2\cdot 3^2|A$, by \Cref{lem:ram1}   (3), $N\in \{5,7,13,19,43\}\cap\{5,13,17\}=\{5,13\}$. 
By \Cref{eg:23513} (2), this is impossible. 
Then by \Cref{lem:ram1} (3),  we know that $A,N$ satisfies  one of the following:

 \begin{itemize}
\item 
  $   A  \text{ divides }2^4\cdot 3,   \quad N=5,7,13,19,43;$ 

\item[] $A \text{ divides } 3^3\cdot 2,  \quad N=5,13,17;$ 
\item[] $   A= 3^2\cdot 7,\quad N= 2;$ 
\item 
  $  A=2^2\cdot 5, \quad N=3,7,13;$ 
\item[] $A \text{ divides }  2^4\cdot 11,\quad N=3;$ 
\item[]   $ ( A,N)=(2^2\cdot 17,3)
,(5^2\cdot3,7)   ;$  

   \item[] $  A=5^2\cdot2, \quad N=3,19.$ 
   \item $  A=2^2\cdot 5^2, \quad N\in \{3,7,13\}\cap\{3,19\}=\{3\}$.
           
\end {itemize}  We  order  the possibilities as above so that it is more  convenient to 
 check them in  Example \ref{eg:ShA10} 
and Example \ref{eg:ShA11}, corresponding to the first two bulletin points   respectively.  
Then by \Cref{lem:auto-critical},  (1.b) follows  (and we do not need to check the third  bulletin point).

       Finally, we consider  $|S_B|>2$. Assume that  $Y_0^A(N)$ is auto-critical, and we want  a contradiction.
       By the discussion above \Cref{lem:ram1},     $ S_B=\{2,3,5,7\},\{2,3,5,11\}$ and  $N=1$. 
          Let $D$ be its discriminant of $B$. 
 By \Cref{lem:auto-critical}, we only need to show that $Y_0^{pD}(1)$ is not auto-critical for   $p\in S_B$. 
Using \textsf{LMFDB}, we find   some
$\Pi_1,\Pi_2,\Pi_3$  appearing in $C$ minimal of    levels $D,{pD},{pD}$ respectively, such that  
 \begin{itemize}
\item 
 $ \Pi_{1,p}=\pm1$ and $ \Pi_{2,p},\Pi_{3,p} $ are minimal of conductors
  $1,2,2$ respectively,
 
  \item for $q\in S_B\bsl\{p\}$, $ \Pi_{1,q}\otimes \Pi_{2,q}{\otimes }\Pi_{3,q}=1$,
  
     \end{itemize} 
Then by \Cref{lem:quat1} and \Cref{lem:Dn} (1) at $p$, 
  $\Pi_1\otimes \Pi_2{\otimes }\Pi_3$ has a  nonzero $B^\times(\BA_\rf)$-invariant linear forms. So $Y_0^{pD}(1)$ is not auto-critical.

   (2)  follows from  (1) and  \Cref{lem:square1}.   
   
   (3) By \Cref{lem:auto-critical}, 
   we only need to check   pairs $(A,N)$  as in (1).
By   \Cref{lem:YKXK} (2),  the natural map  
$X_0^A(N)\to Y_0^A(N)$ is an isomorphism unless   $A= 2^4\cdot 3^2 ,
5^2\cdot 2, 5^2\cdot 3,   7^2\cdot 2$, in which cases $N=1$.  These cases  are 
checked  in Example \ref {eg:ShAX0}.
   
  (4)  The equality for $ A= 3^2\cdot 7,N=2$ is obvious.  For  $ A= 2^2\cdot 5,N=3$, we apply the remark after \Cref{lem:YKXK} and use \textsf{LMFDB}.
 Then by  \Cref{lem:auto-critical},    we only need to  show that for  other pairs $(A,N)$  as in (1.b) such that $(D,N)$ is as in 
   \Cref {thm:Shauto-critical} (3),  $X_1^A(N)$ is not auto-critical.
    This is done in Example \ref {eg:ShAX1}.

(5)    By  \Cref{lem:auto-critical},    we only need to check   pairs $(A,N)$  as in (4) such that $(D,N)$ is as in 
   \Cref {thm:Shauto-critical} (4). However, there is no such  $(A,N)$.
\end{proof}

      \subsection{More general levels}
      
   Even when two curves do not dominate each other, we may still compare  the representations appearing in them.  
    One case is as follows.
Let us first mix some of the above classical  level structures outside $S_B$.   For  coprime  ideals $\fN,\fN'$ of $\cO_F$  that are also coprime to all $v\in S_B$, let $$ X^ \fA(\fN,\fN')=X_K, \ Y^\fA (\fN,\fN')=Y_K=X_K/ \BA_\rf^\times$$  where 
     $  
K_{v}   
$ 
is as in \eqref{KvA}      for $v\in S_B$,
and
             \begin{equation}\label{Kv1}
K_{v}=\Gamma_0\lb\fN\cO_{F_v}  \rb\Gamma\lb\fN'\cO_{F_v}\rb \text{ for } v\not
      \in S_B.
 \end{equation} 
 Then   $
X^ \fA(\fN,\cO_F)=X_0^ \fA(\fN),  Y^ \fA(\cO_F,\fN')=Y^ \fA(\fN') $   and so on.
 We   omit ``$\fA$" if it $B=\RM_{2,\BQ}$.

Now  we have the following   direct consequence of \Cref{Gamma0vsGamma} and \Cref {lem:ep}  (1).
 
       \begin{lem} \label{lem:square1}
 The Shimura curve   $ Y_0^\fA(\fN\fN'^2)$ is auto-critical 
 if the only if  $ Y^\fA(\fN,\fN')$ is auto-critical.
 
    \end{lem} 

\begin{thm}\label{ShNN'auto-critical}
(1) Let $F=\BQ$ and $B\neq \RM_{2,\BQ}$.
 The   Shimura curve $Y^A(N,N')$, $N,N'>1$, is  not auto-critical.
 
 (2) If $B= \RM_{2,\BQ}$,  there are finitely many  pairs of  $N,N'>1$ such that the modular
 curve $Y ^A(N,N')$ is   auto-critical.

(3) (1)(2) also hold  for   $X^A(N,N')$.
   \end{thm}
   \begin{proof}
  (1) 
  By \Cref{lem:auto-critical}
  only need to prove the case that $A$ is square free.  But this  follows from \Cref{thm:Shauto-critical} (1) (and \Cref{readogg} (2)) and    \Cref  {lem:square1}.

  (2) is a special case of 
\Cref  {NN'auto-critical} below , where the auto-critical modular curves are classified.

(3) follows from  \Cref{lem:auto-critical} and (1)(2). 
 \end{proof}

  \begin{thm}\label{NN'auto-critical}
  The   modular curve $Y(N,N')$, $N,N'>1$, is  auto-critical if and only if 
  \begin{itemize}
\item 
$(N,N')= (3,2), (2,3) $, in which case $Y(N,N')$ has genus $ 0$;
 
  \item 
$(N,N')= (5,2), (9,2),(4,3), (7,2)$, in which case $Y(N,N')$ has genus $ 1
 , 1
 , 2
  ,
 2$ respectively;

  \item   $(N,N')= (3,4)$, in which case $Y(N,N')$ has genus $ 6$;

 \item   $(N,N')= (2,5),(8,3)$, in which case $Y(N,N')$ has genus $ 8$. 
 \end{itemize} 
 
  (2) The   modular curve $X(N,N')$, $N,N'>1$, is  auto-critical if and only if 
 $(N,N')  $ is as in (1). Moreover, $X(N,N')=Y(N,N')$, unless $(N,N')=(2,5)$
 in which case $X(N,N')$ has genus $ 16$ and the natural map  
 $X(N,N')\to Y(N,N')$ restricted to each geometrically connected component  is an isomorphism to its image.

 \end{thm}
 \begin{proof} 
 (1)  By \Cref{thm:auto-critical} (1) and \Cref  {lem:square1}, we have the  corollary except the genera.
 Instead of  genus formulas,  the genera can be computed  using \textsf{LMFDB},
 \eqref{eq:H10}
 and   dimension formulas for invariants by compact subgroups of representations of $\GL_2(\BQ_p)$ in  
 \cite[Section 4]{MY}.
 For example, the most complicated case is when $(N,N')= (8,3)$ so that $NN'^2=72$.
 There are $\Pi_1,\Pi_2,\Pi_3 $ (see \cite[Example 3.1.8 (5)]{Qiufinite}) contributing dimensions $3,2,3$ respectively in $H^{1,0}(Y)$ with $Y=Y(8,3)_{\BC}$, so that $g(Y)= \dim H^{1,0}(Y)=8.$   We leave the other cases to the reader.

   (2)  Similar to \Cref{lem:YKXK},  we find  $X(N,N')=Y(N,N')$, unless $(N,N')=(2,5)$.  In this case, we find  its genus $16$ as in (1).
   Thus by \Cref{lem:auto-critical}, (2) holds except that we need to check the auto-criticalness of
   $X_0(2,5)$ and verify the isomorphsm. The auto-criticalness is done in Example \ref{eg:2,5}. 
       The isomorphsm follows as each geometrically connected component  of 
          $X_0(2,5)$ and    $Y_0(2,5)$ has genus $4$  by  \eqref{pi01} below.
      \end{proof}

The   modular curves $X(N,N')$, $Y(N,N')$ are not necessarily geometrically connected. 
Let $G=(\BZ/N')^\times$.
A standard   computation shows that
\begin{equation}\label{pi01}
\pi_0\lb X(N,N')_{\BC}\rb=G,\ \pi_0\lb Y(N,N')_{\BC}\rb=G/G^2.
\end{equation}
For a prime $p|N$, $w_p$ acts on $\pi_0$'s by multiplying $p$. 
Then we have the following examples.

  \begin{eg}\label{eg:auto-critical1} 
 A  connected component $X$ of $Y (3,4)_{\BC}$ is isomorphic to  $Y (3,4)/\pair{w_3}_{\BC}$, and of genus $3$. 
Moreover, $X$ has a smooth affine model 
$$
y^2=x^8 +14x^2 +1
$$ 
so that $X$ is hyperelliptic, and
 $\Aut(X)\cong \PGL_2(\BZ/4)\cong \BZ/2\times S_4$ (read from the database GroupNames \cite{Dok}  for $\GL_2(\BZ/4)$, which 
has  Group ID  $(96,195)$, and its central quotient), which is of order $48$, and  has Group ID  $(48,48)$  in the Small Groups Library. 
  Indeed, 
 by   \cite[Example 3.1.8 (2)]{Qiufinite}  only $\Pi_1$ there appears in $Y (3,4)/\pair{w_3}$ and for $p=2$, $\Pi_{1,p}$ is minimal supercuspidal
 of   conductor $3$. Then by \Cref{prop:k2} (3),
  $\Aut(X)$ contains $ \PGL_2(\BZ/4)$,  and the  quotient of $X$ by the  only nontrivial central element  of  $ \PGL_2(\BZ/4)$ (which is of order $2$) 
 is $\BP^1$. 
 We find the equation of $X$ in  \cite[Table 3]{MSSV}.
\end{eg} 

 \begin{eg}\label{eg:auto-critical2} 
 A  connected component $X$ of $Y (2,5)_{\BC}$ is isomorphic to  $Y (2,5)/\pair{w_2}_{\BC} $, and of genus $4$. 
      Moreover,
 $X$ is   Bring's curve  (see \cite[4.1.2]{QZ1}), which
admits an explicit model
     $$
\left\{     \sum_{i=1}^5 x_i=     \sum_{i=1}^5 x_i^2=    \sum_{i=1}^5 x_i^3=0\right\}\subset \BP^4,
     $$      
    and   $\Aut(X)\cong \PGL_2(\BZ/5) \cong S_5$ acts by permuting the coordinates $x_i$. In particular, $X$ is not hyperelliptic.
  Indeed,  $\Aut(X)$ contains $ \PGL_2(\BZ/5)$ (we omit the verification, which is similar and easier to Example \ref{eg:auto-critical2}), which    is the  largest possible automorphism group  in genus $4$ by \cite[Table 4]{MSSV}.

 %Also note that $J(X)$ is isogenous 
%to $ E^4$, where $E$ is the elliptic curve with minimal Weierstrass equation $y^2+xy+y=x^3+x^2-3x+1$.
%Indeed, there is only one $\Pi$ appearing in $Y (2,5)$. Let $f$ be the corresponding newform. Then  $E$ is the elliptic curve (up to isogeny)  corresponding to this newform, and  can be read from \textsf{LMFDB}.
    \end{eg}  
     \begin{eg}\label{eg:auto-critical3}   A  connected component $X$ of $Y (8,3)_{\BC}$ is isomorphic to  $Y (8,3)/\pair{w_2}_{\BC} $, and of genus $4$. 
  We show  that  $\Aut(X) \cong S_4$, 
  $X$  is non-hyperelliptic,
 and is not a curve in \cite[Section 4]{QZ1}.
 
 Indeed  $\Aut(X)>  \PGL_2(\BZ/3) \cong S_4$ and $|S_4|=24>4(4-1)$ so that $X$ has a large automorphism group in the sense of \cite{MSSV}. %  
 By \cite[Table 4]{MSSV} and checking the automorphism groups there divided by $24$, we find that the only curves of genus $4$ whose automorphism groups  contains $S_4$ (which has Group ID  $(24,12) $ in the Small Groups Library)
form an $1$-dimensional family.
Moreover, only two curves in this family have automorphism groups  larger than $S_4$. We show that $X$ can not to be either of them.
 One of them is Bring's curve in the last example, whose Jacobian of isogenous 
to  to the $4$-th power of an elliptic curve. 
Indeed, by \cite[3.2.2,3.2.3]{YZZ}, this is  information can be read from \textsf{LMFDB}  (and  the dimension counting in the proof of \Cref{NN'auto-critical}). 
The other has CM Jacobian by \cite[6.4]{Wol}.
 However, $J(X)$ is isogenous 
to $ E_1^3\times E_2$, where $E_i$'s are elliptic curves, $E_1$ is non-CM and $E_2$ is CM by $\BQ(\sqrt{-3})$.
 (This is  information can be read from \textsf{LMFDB} as above.)
 Thus,  $X$ can not be  either of them. 
In particular, we know that $X$ is not a curve in \cite[Section 4]{QZ1}, as Bring's curve is the only one of genus $4$ there with automorphism group containing $S_4$.

  Finally, by the  same computation as  in \cite[Section 4]{QZ1} (more precisely, its appendix), no curve in this $1$-dimensional family is hyperelliptic. 
In  particular,  $X$ is non-hyperelliptic. 
    \end{eg}

         \appendix 

\section
 {Trilinear forms}\label{Local root numbers}
  
We need to compute trilinear forms for admissible representations over  local fields, following Prasad  \cite{Pra}. All representations below are assumed to be admissible. 
 \subsection{Notions}\label{Notions}
 
Let  $E$ be a non-archimedean local field with ring of integers $\cO_E$ and residue field $\kappa$. 
Let $\fm_E$ be the maximal ideal of $\cO_E$. 
%For $n\geq 0$, we have $\Gamma_0\lb\fm_E^n\rb,\Gamma_1\lb\fm_E^n\rb,\Gamma\lb\fm_E^n\rb$ as defined in \Cref{1.4}.
Let $V$ be  an irreducible     infinite-dimensional  representation  of $\GL_2(E)$.
Recall that the conductor $\Cond(V)$ of $V$ is  the minimal $n$ such that there exists $ v\in V\bsl0$ such that
the space of invariants
$ V^{\Gamma_1\lb\fm_E^{n}\rb}\neq 0.$ 
 Moreover,  by \cite{Cas}, for $m\geq n$,  \begin{equation}\label{mn}
 \dim V^{\Gamma_1\lb\fm_E^{m}\rb}=m-n+1.
\end{equation}
In particular, if $m=n$, such $v$ is unique up to scaling, and is usually called the new vector of $V$.
Let
 \begin{equation}\label{ALA}
g\in \begin{bmatrix}
0 & 1 \\
\fm_E^{m-1}\bsl \fm_E^{m} & 0.
\end{bmatrix}
\end{equation}
 Then  $
\Gamma_0\lb\fm_E^{m}\rb$ is normalized by  $g$ and 
$g^2$ lies in the center of $\GL(2, E )$. In particular,  if the central character    of $V$ is trivial and $m=n$, then $gv=\pm v$. 
The sign, which is obviously independent of the choice of $g$, is called   the Atkin--Lehner sign of $V$.
For  $m>n$, we need   the following special case of \cite[7.4]{Schm1}.
\begin{lem}  \label {levelup}
  If  the central character   of $V$ is trivial and $m>n$, then  both $\pm 1$-eigenspaces of $g$ is nonzero.

   \end{lem}

An irreducible     infinite-dimensional  representation $V$ of $\GL_2(E)$ is   minimal   if  the conductor $\Cond(V)\leq \Cond(V\otimes \chi) $
for every character $\chi$ of $E^\times$. 
 For example,  the Steinberg representation which we denote by $\St_E$ is  minimal  of   conductor $1$.  (We shall  write $\St$ for  $\St_E$ if $E$ is clear from the context.)
More generally, we have the following  lemma from the classification and conductors of  irreducible   representations, see \cite[Table 1]{LW} and the discussion following Definition 2.5 in loc. cit.. % and the simple fact that $\Cond(V)= \Cond(V\otimes \chi) $ if $\chi$ is an unramified character.
Use $\pm 1$   to denote the only two unramified quadratic character  of $E^\times$, where $-1$ is the nontrivial one.
\begin{lem} \label{minimal}
(1) The   representation $V$ is minimal if  the conductor $\Cond(V)\leq \Cond(V\otimes \chi) $
for every character $\chi$  such that   $\Cond(\chi)\leq \Cond(V)/2$.
In particular, if $\Cond(V)\leq 1$, then $V$ is minimal.

 (2) Assume 
$V$ has an unramified central character $\omega$.  
  If $\Cond(V)=1$, then $V=\St\otimes \chi$ where $\chi$ is an unramified  character such that $\chi^2=\omega$. In this case, 
  $V$ is minimal. Moreover, if $V$ has trivia central character, then $\chi=\pm1 $ and 
    the Atkin--Lehner sign  of $V$ is $\mp1$.

%(3) If $\Cond(V)>1$ and is  odd, then $V$ is minimal supercuspidal.
 
 (3) Assume  $V$  is minimal, has central character $\omega$, and $\Cond(V)>1$.
Then $\Cond(\omega)<\Cond(V)$ 
if and only if $V$ is  supercuspidal.
  
  (4)   Assume 
$V$ is  minimal and has central character $\omega$. Then $\Cond(\omega)=\Cond(V)$ if and only if  $V$ is a principal series.

   \end{lem}
   Here  in (2),   the extra information on  the Atkin--Lehner sign is computed in \cite[3.1.2]{Schm}. 
   %   In (4), the minimality of $V$ in fact can be removed. However, we do not need this generality. %V is twistof  a minimal one. 

\begin{lem} [{\cite[Proposition 2.1, 2.2, 4.3]{MY}}] \label {Gamma0vsGamma}
   Assume that $V$ is not a principal series or $V$ has trivia central character.  
Then $\Cond(V)\leq 2n$ if and only if 
$
 \dim V^{\Gamma \lb\fm_E^{n}\rb}\neq 0.
$

   \end{lem}

Let $D$ be the unique division quaternion algebra over $E$ and $\nrd:D\to E$ the reduced norm map.
Let $ O=\cO_D:=\nrd^{-1}( \cO_E)$, which  the unique maximal $\cO_E$-order in $D$, and  $\fm_D=\fm_D:=\nrd^{-1}( \fm_E)=\cO_D\bsl \cO_D^\times$
the unique maximal two-sided   ideal, as well as the unique left or right ideal.

For a  finite-dimensional irreducible representation $W$ of $ D^\times,$    its conductor $\Cond(W)$ is the minimal $n$ such that 
$W|_{1+\fm_D^{n-1}}$ is trivial.
Call $W$    minimal   if  $\Cond(W)\leq \Cond(W\otimes \chi) $ for every character $\chi$ of $E^\times$.
   \begin{lem}[{\cite[Lemma 6.5]{Pra}}] \label{lem:quat1} 
 Let $|\kappa|=q$ and let $W$ be minimal.  If $\Cond(W)=2n+1>1$ (resp. $\Cond(W)=2n>0$), then 
$\dim W=q^{n-1}(q + 1)$    (resp. $\dim(W)=2q^{n-1}$). 
\end{lem}

  Let $V$  be a   discrete series   of $\GL_2(E)$   and $\JL(V)$  its Jacquet--Langlands correspondence to $ D^\times.$  Then  they have the same  central character. And 
     \begin{equation}
     \Cond(V)=\Cond(\JL(V))
     .\label{JLcond}
   \end{equation}

Since $\JL(V\otimes \chi)=\JL(V)\otimes \chi$ for any character $\chi$, the Jacquet--Langlands correspondence preserves minimality. 
We  also recall that
   $\JL(\St )= 1$,   the  trival representation of  $ 
  D^\times $. So
   \begin{equation}
 \JL(\St \otimes \chi)= \chi \circ \nrd .\label{JLchi}
   \end{equation}
  Combining it with \Cref{minimal}, we have the following.
  \begin{cor}\label{JLpm1}
The  Jacquet--Langlands correspondence of an irreducible     infinite-dimensional  representation $V$  of $\GL_2(E)$
 is  $\chi \circ \nrd$  of $D^\times $ where $\chi
 $ is an unramified character if and only if  $V$ has unramified   central character and conductor $1$.
 The unramified character   $\chi$
  is $\pm1$  if and only if  $V$  has Atkin--Lehner sign $\mp 1$.
 \end{cor}
More generally, we have  the   character identity  \begin{equation}\label{character identity2} 
 \ch(V)(x) = -\ch(\JL(V'))(x')
    \end{equation}
  at $x\in \GL_2(E)$ elliptic regular semisimple  and $x'\in D^\times$ regular semisimple with the same characteristic polynomial. 
 
   \subsection{Prasad's criteria} 
Now we start to discuss trilinear forms following Prasad \cite{Pra,Pra1}.

The following lemma is obvious.
\begin{lem}\label{lem:cc}
Let $V_1,V_2,V_3$  be   irreducible    representations of $G=\GL_2(E)$ or $D^\times$.
  If   the product of the  central  characters  of  $V_1,V_2,V_3$ is not trivial, then  
   $$  \Hom_{G}\lb V_1\otimes V_2\otimes V_3,\BC\rb=0$$

    \end{lem}
  
   The main results of  \cite{Pra,Pra1} are as follows

 \begin{thm}\label{thm:Pra}
  Let $V_1,V_2,V_3$  be irreducible  infinite-dimensional
   representation of $ \GL_2(E)$ such that
 the product of their central  characters is trivial. Then
 $$ \dim \Hom_{\GL_2(E)}\lb V_1\otimes V_2\otimes V_3,\BC\rb+  \dim \Hom_{D^\times}\lb \JL(V_1)\otimes \JL(V_2)\otimes \JL(V_3),\BC\rb=1.$$
 Here if $V_i$ is not a  discrete series, $\JL(V_1)$ is understood as $0$.

\end{thm}
The group admitting a nonzero trilinear form is also determined bu the
 triple product local root number. We will not discuss it in this paper.

 Let $V_1,V_2,V_3$  be irreducible  infinite-dimensional
   representation of $ \GL_2(E)$ such that
 the product of their central  characters is trivial.   
\begin{lem}\label{lem:ep}
 (1) If  some $V_i$ is a principal series,  then $$  \Hom_{\GL_2(E)}\lb V_1\otimes V_2\otimes V_3,\BC\rb\neq 0.$$

        (2) Assume $V_i=\St\otimes \chi_i$ for $i=1,2$ where $\chi_i$ is a character. Then   $$  \Hom_{\GL_2(E)}\lb V_1\otimes V_2\otimes V_3,\BC\rb =  0$$  if and only if $V_3=\St\otimes \chi_3$ with 
        $\chi_1\otimes\chi_2\otimes \chi_3=1$.     
                      
      (3) Assume   $V_1,V_2$ are  discrete series   and 
      $V_3=\St $, then  $$  \Hom_{\GL_2(E)}\lb V_1\otimes V_2\otimes V_3,\BC\rb = 0$$ if and only if $V_1,V_2$ are  dual to each other.

  (4) If   $V_1,V_2,V_3$ are  discrete series  of conductors $\Cond(V_1)\leq \Cond(V_2)<\Cond(V_3)$, then  $$  \Hom_{\GL_2(E)}\lb V_1\otimes V_2\otimes V_3,\BC\rb\neq 0.$$
  \end{lem}
 \begin{proof}  
 The criteria  can be easily deduced from  \Cref{thm:Pra} and the corresponding (under the  Jacquet--Langlands correspondence) trilinear form computations  on $D^\times$. For (4), $\JL(V_1)\otimes \JL(V_2)$ is a direct sum of representations of conductors at most $\Cond(\JL(V_2))$.  So
  $$  \Hom_{D^\times}\lb \JL(V_1)\otimes \JL(V_2)\otimes \JL(V_3),\BC\rb=0.$$
 \end{proof}

A  maximal compact-modulo-center subgroup of $\GL_2(E)$ is 
 one of the following two, which are called unramified and ramified respectively:  
  \begin{equation}\label{cK} 
  \cK_E^\ur=E^\times\GL_2(\cO_E),\quad \cK_E^\ram=E^\times \pair{\Gamma_0 \lb\fm_E\rb,g},
    \end{equation}
     where $g$ is as in \eqref{ALA}  with $m=1$.
If $E$ is clear from the context, we omit the subscript. 
   \begin{thm}
   [{\cite[Theorem 6.1]{Pra}}]
   \label{thm:type}
  A minimal irreducible supercuspidal representation $V$ of $\GL_2(E)$ 
  is of the form  
 $ V=\Ind_\cK^{\GL_2(E)}M$, where $\cK$  is a maximal compact-modulo-center subgroups of $\GL_2(E)$
 and $M$ is a very cuspidal representation\footnote{See \cite[p 21]{Pra} for the definition of  a very cuspidal representation of $\cK$.} of $\cK$ of
 level $[\Cond (V)/2 ]$. % and 
%  dimension $q^{[\Cond (V)/2]-1}(q-1)$, 
%where $|\kappa|=q$.   
Moreover  if $\Cond(V)$ is even, then $\cK= \cK ^\ur$;
 if $\Cond(V)$ is odd, $\cK=\cK^\ram$.%Note that in odd case W|\Gamma_0(\fm_E) is irred  by the same proof as {prop:k2} (3).

\end{thm}

  The following result is a consequence  of \Cref{thm:Pra} and   \cite[Proposition 6.7]{Pra}.
  (However, it   can also be deduced  directly from the statement in the first sentence after the proof of \cite[Proposition 6.7 ]{Pra}.) 

   \begin{lem}\label{scusp}   
Let $V_i=\Ind_\cK^{\GL_2(E)}M_i$ be  minimal supercuspidal representations as in \Cref{thm:type}
for the same $\cK$ (so that they are  of the same conductor). Then 
 $$\dim \Hom_{\GL_2(E)}\lb\mathop\otimes_{i=1}^3V_i,\BC\rb =\dim \Hom_{\cK}\lb \mathop\otimes_{i=1}^3 M_i,\BC)\rb.$$

\end{lem}

 In the rest of this subsection, we consider the computational aspects of Jacquet--Langlands correspondence and trilinear forms.
 Let $\varpi\in \cO_E$  be a uniformizer. 
For  non-negative integer $m,n$,  $ D^\times/\varpi^{m\BZ}({1+\fm_D^{n}})$ is an obviously computable finite group. One can represent it as a successive extensions by finite abelian groups. In \Cref{Quaternion algebra}, this is done   for $n=2$ with $\varpi^{m\BZ}$ replaced by $E^\times$. See  also \cite[A.4]{Qiufinite}. 
 The computation of trilinear forms on $ D^\times$  can be done using characters as usual. 
 
  \begin{lem}\label{localalgoD}   
Let $ W_i$, $i=1,2,3$,  be     irreducible admissible representations of $ D^\times$, then $  \Hom_{D^\times}\lb W_1{\otimes }  W_2\otimes W_2,\BC\rb$ is computable. 
\end{lem} 

For $\GL_2(E)$, we can transfer the  computation of trilinear forms  to the 
 computation on $ D^\times$ via Jacquet--Langlands correspondence in the discrete case (while the case involving principal series is covered in \Cref{lem:ep} (1). 

  The Jacquet--Langlands correspondences of special representations is covered in \Cref{JLchi}. We consider supercuspidal ones. 
The relation between a minimal irreducible supercuspidal representation   $ V=\Ind_\cK^{\GL_2(E)}M$ of   $\GL_2(E)$ 
 and $M$ is given by   the   character identity
  \begin{equation}\label{character identity1} 
 \ch(V)(x) = \ch(M)(x)
    \end{equation}
 for   ``generic"  $x\in \cK\bsl E^\times \cK(m) $, where $m = [\frac{\Cond(V)}{2}]$ is the level of $\cK$ and $\cK(m)$ is the level $n$ principal congruence subgroup of $\cK$.  We refer to \cite[Lemma 6.3]{Pra} for the precise conditions on $x$.
As a consequence of the character identities \eqref{character identity2} and \eqref{character identity1}, we have the following.

  \begin{prop}\label{JLalgo}   
  Given the conductor $n$ and the compact induction data  $(\cK,W)$ of   $V$,   its Jacquet--Langlands correspondence $\JL(V)$ to $ D^\times$  is a computable representation of $ D^\times/({1+\fm_D^{n-1}})$. 
     
\end{prop} 

Combined with \Cref    {thm:Pra} and \Cref{localalgoD}, we have the following.
   \begin{cor}\label{localalgo}   

  Let $V_1,V_2,V_3$  be irreducible  infinite-dimensional
   representation of $ \GL_2(E)$.  
   Then
 $ \dim \Hom_{\GL_2(E)}\lb V_1\otimes V_2\otimes V_3,\BC\rb$ is computable.

\end{cor}

       \subsection{Quaternion algebras}
\label{Quaternion algebra}
          We study the structures of finite quotients of $D^\times$. Then use the character tables of these finite groups  to compute trilinear forms.
          
          Let $E'$ be the unique unramified (separabe) quadratic extension  of $E$, which is contained in 
 $D$ as an $F$-subalgebra. 
Recall  that $\fm_D$ is principal and let $\varpi_D$ be a generator.
Then 
    \begin{equation*}\label{DZ/2} D^\times/ O^\times\cong \BZ,\quad
    D^\times/E^\times O^\times\cong \BZ/2 ,
   \end{equation*}
   and both are generated by $\varpi_D$.
 Moreover, 
   we can choose $\varpi_D$ such that  $\varpi_D^2\in E$ (so $\varpi_D^2$ is a generator of $\fm_E$) and
    the conjugation by $\varpi_D$ is the nontrivial element in  $\Gal(E'/E)$.

First, we consider   $D^\times/E^\times(1+\fm_D)$,
using the exact sequence
  $$
  1\to O^\times/ O_E^\times(1+\fm_D)\to D^\times/E^\times(1+\fm_D) \to D^\times/E^\times O^\times\cong \BZ/2\to 1,
  $$
  which is  split by sending the generator of $\BZ/2$ to $\varpi_D$. 
Let $  \kappa':=O / \fm_D $, which is also the residue field of $E'$, as well as the unique separable  quadratic extension of the residue field $\kappa$ of $E$. 
Then  $ O^\times/ O_E^\times(1+\fm_D)\cong \kappa'^\times/\kappa^\times$. 
This
  gives
   $$
   D^\times/E^\times(1+\fm_D)\cong \kappa'^\times/\kappa^\times\rtimes \BZ/2,
   $$
  where $\BZ/2$ acts on $ \kappa'$ by $\Gal(\kappa'/\kappa)$.
   Let $  \tau$ be the  generator
         of $\Gal(\kappa'/\kappa)$, then $\tau(a)=a^{|\kappa|}=a^{-1}$. 
         
Second,  we consider  $D^\times/E^\times(1+\fm_D^2)$.  Since $(1+\fm_D)\cap E^\times\subset  1 + \fm_D^2 $, we have
   the exact sequence
  $$
  1\to ( 1 + \fm_D ) / ( 1 + \fm_D^2 ) \to D^\times/E^\times(1+\fm_D^2) \to D^\times/E^\times(1+\fm_D)\cong \kappa'^\times/\kappa^\times\rtimes \BZ/2\to 1.
  $$
  which is  split by  choosing $\varpi_D$ and using the Teichm\"uller lifting   of $\kappa'^\times$ (i.e., 
  group of roots of unity in $D^\times$,  which is also the  group of roots of unity in  $E'^\times$).
Since the isomorphisms
   \begin{alignat*}{3}
 \kappa'=  O / \fm_D& \cong \ \ \ \ \  \fm_D / \fm_D^2 && \cong( 1 + \fm_D ) / ( 1 + \fm_D^2 )  ,\\
   x \mod \fm_D&\mapsto  \varpi_D x\mod \fm_D^2&&\mapsto  1+\varpi_D x \mod 1+\fm_D^2
     \end{alignat*}
     are equivariant under conjugation by $\varpi_D$, 
         $\BZ/2$ acts on  $( 1 + \fm_D ) / ( 1 + \fm_D^2 ) \cong\kappa' $ by  $\Gal(\kappa'/\kappa)$.       For the action of  
      $ a\in
      \kappa'^\times/\kappa^\times \cong O^\times/ O_E^\times(1+\fm_D) $, 
      let $\wt a\in O^\times$ be a lift and let   $ \wt  \tau$ be the  generator
         of $\Gal(E'/ E)\cong \Gal(\kappa'/\kappa)$. 
      Then 
      $$\wt a (1+\varpi_D x)\wt a^{-1}= 1+\wt a\varpi_D x \wt a^{-1}=1+\varpi_D \wt \tau(\wt a) x\wt a^{-1}  .$$
      Thus letting $ \kappa'^1 \subset  \kappa'^\times$ be the subgroup of elements of norm $1$, then
      $ \kappa'^\times/\kappa^\times$ acts on 
       $( 1 + \fm_D ) / ( 1 + \fm_D^2 ) \cong\kappa' $ 
       via  \begin{align*} 
     \kappa'^\times/\kappa^\times&\cong   \ \ \  \kappa'^1 \ \ \ \subset  \kappa'^\times  \subset \Aut(\kappa' ), \\
          a&\mapsto    \tau(  a)    a^{-1}  
     \end{align*}
            where $\kappa'^\times  $ acts on $\kappa' $
       by field multiplication.  
       In conclusion,
       $$D^\times/E^\times(1+\fm_D^2)\cong \kappa' \rtimes\lb  \kappa'^\times/\kappa^\times\rtimes \BZ/2\rb ,$$
       with the action of as above.        In particular, we have the following.
  
      \begin{lem} \label{lem:quat}(1) If  $|\kappa|=q$, then $D^\times/E^\times (1+\fm_D)\cong D_{q+1}$, the dihedral group of order $2(q+1)$.
      And $D^\times/E^\times (1+\fm_D^2)\cong C_q^2\rtimes D_{q+1}$,  where $C_q=\BZ/q$ and   $D_{q+1}\to \Aut(C_q)$ is injective. 
      
(2) If $|\kappa|=2$, $D^\times/E^\times (1+\fm_D^2)\cong S_4$, the symmetric  group of four elements,  and it has Group ID  $(24,12) $ in the Small Groups Library.

 \end{lem}  
 \begin{proof}  We only need to prove the injectivity part of (1), which follows from a direct computation. 
  \end{proof}

  Let us draw two   lemmas  and their further corollaries  from the last lemma. 
First, by  inspecting  the character table of $D_{n}$, we have the following. (Note that   the characters of 
$D^\times/E^\times(1+\fm_D^n)$
are all real as the corresponding representations are self-dual.)
   
     \begin{lem} \label{lem:Dn}(1) Let $W_1,W_2$  be  irreducible $2$-dimensional representations   of $D^\times/E^\times (1+\fm_D)$. 
    Assuming $\chi=\pm1$, then
     $$ \Hom_{D^\times}\lb W_1{\otimes }  W_2\otimes \chi,\BC\rb  \neq 0 $$ if and only if $W_1=W_2$.  Moreover, if $|\kappa|=3$,   the same is true replacing   $\chi=\pm1 $  by  any character of $D^\times/E^\times  $.
  
 %(Note that if $|\kappa|=2,3$, there is only one     irreducible $2$-dimensional representation so that $W_1=W_2$.)}
          
(2)  If $|\kappa|=%3,4,7,9,13
3,7,13$, for any irreducible $2$-dimensional representation $W$ of $D^\times/E^\times (1+\fm_D)$, $$ \Hom_{D^\times}\lb W^{\otimes 3},\BC\rb=0.$$

(3)
   If  $|\kappa|=5$, $D^\times/E^\times (1+\fm_D)\cong D_6$  has a unique nontrivial central element  and only two irreducible  representations $W_+,W_-$ of   
   of dimension $2$. Moreover, the nontrivial central element  acts on $W_\pm$ as $\pm1$,  
   and $$ \Hom_{D^\times} \lb W _+^{\otimes3},\BC\rb\neq 0,\ \Hom_{D^\times} \lb W _-^{\otimes3},\BC\rb= 0.$$
Finally, the twist of $W_\pm$  by   any quadratic  unramified character is still $W_\pm$, and 
  $W_+$ is a twist of $W_-$ by the any quadratic  character of conductor 1.
 \end{lem}

   \begin{cor} \label{cor:Dn}
   (1)   Let  $V_1,V_2$  be     minimal irreducible     infinite-dimensional  representations   of $\GL_2(E)/E^\times$ 
of  conductor $2$ and       $V_3=\St \otimes\chi$.  Assuming $\chi=\pm1$, then  
$$  \Hom_{\GL_2(E)}\lb V_1\otimes V_2\otimes V_3,\BC\rb =  0.$$
   if and only if $V_1=V_2$.    Moreover, if $|\kappa|=3$,   the same is true replacing   $\chi=\pm1 $  by  any character of $\GL_2(E)/E^\times  $.

   (2)  If $|\kappa|=3,7,13$ and $V$ is a minimal irreducible     infinite-dimensional  representation   of $\GL_2(E)/E^\times$ %and $V^-$ the Galois conjugate (indeed, $V^+=V^-$ if   $|\kappa|=3$)
of conductor $2$, then $$  \Hom_{\GL_2(E)}\lb V^{ \otimes 3},\BC\rb \neq  0.$$ 

        (3)  Let $|\kappa|=5$. There are only two minimal irreducible     infinite-dimensional  representations   of $\GL_2(E)/E^\times$ %and $V^-$ the Galois conjugate (indeed, $V^+=V^-$ if   $|\kappa|=3$)
of conductor $2$.
Let $V$ be one of them, which is the compact inductions of  a   very cuspidal representation  $M$  of $\cK^{\ur}$.
Then the action of $ \begin{bmatrix}
0 & 1 \\
2 & 0
\end{bmatrix}\in \cK^{\ur}$ on $ M$    has trace $\pm 2$. 
Moreover,   $$  \Hom_{\GL_2(E)}\lb V^{ \otimes 3},\BC\rb \neq 0  $$ 
  if and only if the trace is $+2$.
  \end{cor}  
 
 \begin{proof}By \Cref{minimal} (3), the  minimal irreducible     infinite-dimensional  representations   of $\GL_2(E)/E^\times$ 
of  conductor $2$ are supercuspidal. 
Then by  \Cref{lem:quat1} and \Cref{thm:Pra}, 
 (1) and (2) follow from \Cref{lem:Dn} (1) and (2) respectively.
We prove (3). 

Consider the element  $g \in D$ that has the same minimal polynomial as $ \begin{bmatrix}
0 & 1 \\
2 & 0
\end{bmatrix}$, that is, $x^2-2$. Its image $\ol g$ in $D^\times/E^\times (1+\fm_D)  \cong D_6$ is the unique nontrivial central element. 
Then
by  the character identities for the Jacquet--Langlands correspondence and compact induction  (see \eqref{character identity2} and \eqref{character identity1}),   the action of $ \begin{bmatrix}
0 & 1 \\
2 & 0
\end{bmatrix}$ on $ M$  
  has trace $\pm 2$ if and only if  the Jacquet--Langlands correspondence  of $V$ to  $D^\times$ is $W_{\mp}$
defined in \Cref {lem:Dn} (3). Now (3) follows from \Cref{thm:Pra} and \Cref {lem:Dn} (3).
 \end{proof}
 
 \begin{rmk} One may also prove (2) and (3) of the corollary directly using \Cref{scusp} and the character table of $\PGL_2(\BF_{|\kappa|})$.  
Actually, using a general fact  about the character table of $\PGL_2(\BF_{|\kappa|})$ in \cite[Corollary 3.2.3]{QZ1}, one can prove that (2) holds if 
$|\kappa| \not \equiv -1\pmod 3$.

 \end{rmk}
 Second,    by  inspecting  the character tables of  $D^\times/E^\times (1+\fm_D^2)$ for $|\kappa|=2,3,5$, we have the following.   
     \begin{lem} \label{lem:Dn1}(1)  
      Let  $|\kappa|=2$ and $W_1,W_2,W_3$  irreducible  representations  of $D^\times/E^\times (1+\fm_D^2)$ of dimensions $2$ or $3$.
     Assume $\lb \dim W_1,\dim W_2,\dim W_3\rb\neq(2,2,3)$  up to permutation. Then  
$$ \Hom_{D^\times}\lb  W_1\otimes W_2\otimes W_3,\BC\rb\neq 0.$$

(2)  
   Let  $|\kappa|=3$ and $W_1,W_2 $  irreducible  representations  of $D^\times/E^\times (1+\fm_D^2)$ of dimensions $2$ and  $4$ respectively.
 Then  
$$ \Hom_{D^\times}\lb  W_1\otimes  W_2^{\otimes 2},\BC\rb\neq 0.$$

(3)  
   Let  $|\kappa|=5$ and $W $ an irreducible  representation  of $D^\times/E^\times (1+\fm_D^2)$ of dimension $6$.
 Then  
$$ \Hom_{D^\times}\lb    W^{\otimes 3},\BC\rb\neq 0.$$

 \end{lem}   
 
Then similar to \Cref{cor:Dn}, we have the following.

     \begin{cor} \label{cor:Dn1}
    Let  $|\kappa|=2$ and let   $V_1,V_2,V_3$  be minimal irreducible     infinite-dimensional  representations   of $\GL_2(E)/E^\times$ 
of conductors $2$ or $3$. 
     Assume $\lb \Cond V_1,\Cond V_2,\Cond V_3\rb\neq(2,2,3)$  up to permutation. Then  
$$  \Hom_{\GL_2(E)}\lb V_1\otimes V_2\otimes V_3,\BC\rb =  0.$$

      %(2) Let  $|\kappa| 3$ and  $V _1,V_2$  minimal irreducible     infinite-dimensional  representations   of $\GL_2(E)/E^\times$ 
%of conductors $2,3$  respectively. 
 %   Then  
%$  \ep \lb V \otimes V_2\otimes V_2\rb=-1$.

 \end{cor}

    Finally, we consider a minimal supercuspidal representation $V$ of $\PGL_2(E)$ of conductor $3$ when $|\kappa|=2$.  Recall that  the Jacquet--Langlands correspondence preserves conductor and minimality.
         By   \Cref {lem:quat1} and \Cref{lem:quat}, the Jacquet--Langlands correspondence of  $V$ is a $3$-dimensional representations of $D^\times/E^\times (1+\fm_D^2)\cong S_4$. 
         There are two  $3$-dimensional representations of $D^\times/E^\times (1+\fm_D^2)\cong S_4$, and non of them is a twist of a representation $D^\times/E^\times (1+\fm_D)\cong S_3$, i.e., both are minimal.
         Indeed, the only representations of $  S_3$  are $1 $, sign representation which we denote by $-1$ for later convenience, and  the unique irreducible representation of dimension $2$, which we denote by $\rho$.   
                  
                           We also consider the  $\Gamma(\fm_E^2)$-invariants $V^{\Gamma(\fm_E^2)}$, as a representation of $${\PGL_2(\cO_E)/\Gamma(\fm_E^2)}\cong \PGL_2(\BZ/4)\cong \BZ/2\times S_4$$ (the last isomorphism is read from the database GroupNames \cite{Dok} for $\GL_2(\BZ/4)$, which 
has  Group ID  $(96,195)$, and its central quotient). By \cite[Theorem 3.5]{LW}, $V^{\Gamma(\fm_E^2)}$ is $3$-dimensional irreducible.  
         Note that the center $\BZ/2<\PGL_2(\BZ/4)$ is generated by $c:=\begin{bmatrix}
1 & 2 \\
2 &  -1
\end{bmatrix} $.  
From the character tables of $\PGL_2(\BZ/4)$ (see \cite{Dok} again), we know that $V^{\Gamma(\fm_E^2)}$ is faithful if and only if $c$ acts by $-1$ on it.
                   
 \begin{prop}\label{prop:k2} 
 (1) Let  $|\kappa|=2$. 
 There is a unique minimal supercuspidal representation $V^{\pm}$ of $\PGL_2(E)$ of conductor $3$ and Atkin--Lehner sign $\pm1$.
 
(2) The Jacquet--Langlands correspondence of  $V^{\pm}$  to $D^\times/E^\times $ is the unique  $3$-dimensional representation of $D^\times/E^\times (1+\fm_D^2)\cong S_4$ whose restriction to $  S_3$  is $\mp 1\oplus \rho$. 
Here we recalls that all the embeddings  $S_3\incl S_4$ are conjugate to each other.

(3)  For $V =V^{\pm}$,   $c\in \PGL_2(\BZ/4)$  acts by $-1$ on $V^{\Gamma(\fm_E^2)}$ so that  $V^{\Gamma(\fm_E^2)}$ is  a faithful 
representation of $\PGL_2(\BZ/4)$.
     \end{prop}
      \begin{proof} 
   (1)   There are two minimal supercuspidal representations of $\PGL_2(E)$ of conductor $3$ by the discussion above. 
     Let us give the  construction of $V^\pm$ and then (1) is proved.
   Recall $\cK^{\ram}=E^\times \pair{\Gamma_0 \lb\fm_E\rb,g}$ where $g$ is as in \eqref{ALA}  with $m=1$. 
      Let $G\subset \Gamma_0 \lb\fm_E\rb$ consist of $\begin{bmatrix}
x & y \\
2z & w 
\end{bmatrix}$
with $y+z\in 2\cO_E$. Then $G$ is a normal subgroup of ${\cK^\ram}$ such that ${\cK^\ram}/E^\times G\cong (\BZ/2)^2$, generated by 
$c$ and $g$. 
      Let $\eta^{\pm}$ be the two character of ${\cK^\ram}/E^\times G $  such  that $\eta^{\pm}(c)=-1,\eta^{\pm}(g)=\pm1$. Then 
$    V^{\pm}:=  \Ind_{\cK^\ram}^{\PGL_2(E)}\eta^{\pm}$ is minimal supercuspidal of conductor $3$ by \Cref{thm:type}. 

Claim: the  Atkin--Lehner sign of $    V^{\pm}$ is $\pm 1$. Let $d:=\begin{bmatrix}
2 & 0 \\
0 &  1
\end{bmatrix}$. 
Note that 
$d \Gamma_0 \lb\fm_E^3\rb d^{-1}\subset G$. So the unique up to scalar vector (function) in $      \Ind_{\cK^\ram}^{\PGL_2(E)}\eta^{\pm}$ supported on 
${\cK^\ram} d $ is $ \Gamma_0 \lb\fm_E^3\rb$-invariant, and so is  the new vector.
Then the claim follows by a direct computation.

(2) By the character identities for  the Jacquet--Langlands correspondence and compact induction  (see \eqref{character identity2} and \eqref{character identity1}), $\eta^{\pm}(g)$  is the \textit{negative} to the trace of    $\varpi\in D^\times/E^\times (1+\fm_D^2)$ on $\JL\lb  V^{\pm} \rb$. 
So the latter trace is $\mp1$.
 Note that $\varpi  $ is contained in some 
embedding  $S_3\incl S_4$ (as we have discussed above \Cref{lem:quat}), and is the unique order $2$ element of $S_3$. From the character tables of $S_4,S_3$, (2) follows.

(3) Recall   that 
by \cite[Theorem 3]{Cas1} (more precisely its proof), $   \Ind_{\PGL_2(\cO_E)}^{\PGL_2(E)} V^{\Gamma(\fm_E^2)}=V^+\oplus V^-$.
Note that by the discussion in the proof of (1), $V^+\oplus V^-=      \Ind_{  \Gamma_0 \lb\fm_E\rb }^{\PGL_2(E)}\eta$, where $\eta$ is the unique nontrivial character of $ \Gamma_0 \lb\fm_E\rb/G$.  
By Mackey's formula and dimension reason, as a representation of $\PGL_2(\cO_E)$, $ V^{\Gamma(\fm_E^2)}\cong  \Ind_{\Gamma_0 \lb\fm_E\rb}^{\PGL_2(\cO_E)}\eta$.  As $\eta(c)=-1$, $c$ as on $ V^{\Gamma(\fm_E^2)}$ as $-1$.
     \end{proof}

\end{document}